\documentclass[10pt]{amsart}
\usepackage{amsmath,amssymb, amsfonts}
\usepackage[all]{xy}
\usepackage{color}
\usepackage{verbatim}
\usepackage{enumerate}
\usepackage{tikz-cd}
\usepackage{MnSymbol}

\DeclareMathAlphabet{\mathpzc}{OT1}{pzc}{m}{it}
\newcommand{\adj}[4]{#1\negmedspace: #2\rightleftarrows #3:\negmedspace #4}
\newtheorem{thm}{Theorem}[section]
\newtheorem{cor}[thm]{Corollary}
\newtheorem{example}[thm]{Example}
\newtheorem{lem}[thm]{Lemma}

\newtheorem{prop}[thm]{Proposition}
\theoremstyle{definition}
\newtheorem{defn}[thm]{Definition}
\newtheorem{rem}[thm]{Remark}
\numberwithin{equation}{section}
\newtheorem{notation}[thm]{Notation}

\DeclareFontFamily{U}{rsf}{} \DeclareFontShape{U}{rsf}{m}{n}{
  <5> <6> rsfs5 <7> <8> <9> rsfs7 <10->  rsfs10}{}
\DeclareMathAlphabet{\mathscr}{U}{rsf}{m}{n}

\newcommand*{\defeq}{\mathrel{\vcenter{\baselineskip0.5ex \lineskiplimit0pt
                     \hbox{\scriptsize.}\hbox{\scriptsize.}}}%
                     =}

\renewcommand{\imath}{\sqrt{-1}}

\DeclareMathOperator{\colim}{colim}

\DeclareMathOperator{\Ban}{Ban}

\DeclareMathOperator{\sNorm}{sNorm}

\DeclareMathOperator{\Ho}{Ho}

\numberwithin{equation}{section}

\begin{document}

\title[]{A note on deriving unbounded functors of exact categories, with applications to Ind- and Pro- functors}
\author{Jack Kelly}\thanks{}%
\address{Jack Kelly,
The Hamilton Mathematics Institute,
School of Mathematics,
Trinity College Dublin,
Dublin 2,
Ireland}
\begin{abstract}
In this note we show that under very mild conditions on a functor between exact categories $F:\mathpzc{D}\rightarrow\mathpzc{E}$ it is possible to derive $F$ at the level of unbounded complexes. We also give applications to deriving functors between $Pro$- and $Ind$- categories.
\noindent 
\end{abstract}
\maketitle 
\tableofcontents
Let $F:\mathpzc{D}\rightarrow\mathpzc{E}$ be a right-exact functor of abelian categories. Using Spaltenstein's seminal work on constructing unbounded resolutions in abelian categories \cite{Spaltenstein}, one can show (\cite{stacks-project} Chapter 13, Proposition 29.2) that whenever filtered colimits over the countable ordinal $\omega$ exist and are exact in both $\mathpzc{D}$ and $\mathpzc{E}$, and $F$ commutes with such colimits, then there is an unbounded left derived functor
$$\mathbb{L}F:D(\mathpzc{D})\rightarrow D(\mathpzc{E})$$
Using a version of Spaltenstein's result for exact category which we proved in \cite{kelly2016homotopy}, in this note we will establish the following

\begin{thm}
Let $F:\mathpzc{D}\rightarrow\mathpzc{E}$ be a functor between weakly idempotent complete, weakly $(\omega,\textbf{AdMon})$-elementary exact categories. If $\mathpzc{D}$ has kernels, and there exists an $F$-projective subcategory $\mathcal{S}\subset\mathpzc{D}$, and $F$ commutes with countable coproducts of objects in $\mathcal{S}$, then there is an unbounded left derived functor
$$\mathbb{L}F:D(\mathpzc{D})\rightarrow D(\mathpzc{E})$$
\end{thm}
To prove this we will show that any weakly $(\omega,\textbf{AdMon})$-elementary exact category embeds in an abelian category which is itself weakly $(\omega,\textbf{AdMon})$-elementary.

We then explain that in many exact categories, derived colimit/ derived limit functors can be defined at the level of unbounded complexes and relate this to deriving functors involving categories of Ind objects. We give dual results for right derived functors. As applications, we consider unbounded derived functors of categories involved in functional analysis, such as the categories $\mathpzc{Born}_{k}$ and $\mathcal{T}_{c,k}$ of, respectively,  bornological $k$-vector spaces and locally convex topological $k$-vector spaces over a Banach field $k$. For example we show that an exact functor $F:sNorm_{k}\rightarrow\mathpzc{D}$ from the category  of semi-normed $k$-vector spaces to a nice enough exact category extends to a right derivable functor of unbounded complexes $Ch(\mathcal{T}_{c,k})\rightarrow Ch(\mathpzc{D})$:

\begin{prop}
Let $k$ be a Banach field, $\mathpzc{D}$ an exact category, and  $F:\mathcal{T}_{c,k}\rightarrow\mathpzc{D}$ a functor whose restriction to the category $\sNorm_{k}$ is exact. Suppose that $\mathpzc{D}$ is ML (Definition \ref{defn:ML}). Then there is a right derived functor
$$\mathbb{R}F:Ch(\mathcal{T}_{c,k})\rightarrow Ch(\mathpzc{D})$$
Moreover if $\mathpzc{D}$ is strongly ML, then for $X\in\mathcal{T}_{c,k}$, $\mathbb{R}F(X)$ is quasi-isomorphic to a complex concentrated in degrees $0,-1$.
\end{prop}

This proposition applies in particular to the functor $(-)^{b}:\mathcal{T}_{c,k}\rightarrow\mathpzc{Born}_{k}$ to the category of bornological $k$-vector spaces of convex type, which endows a space $X$ with its von Neumann bornology.

Throughout the paper unless stated otherwise we shall assume that all exact categories are weakly idempotent complete. Moreover our complexes will have homological grading, i.e. if $X_{\bullet}$ is a complex then $d_{n}:X_{n}\rightarrow X_{n-1}$. For $\mathpzc{E}$ an additive category, $\mathcal{K}(\mathpzc{E})$ will denote the homotopy category of complexes.

\subsection*{Acknowledgements}
The research of the author was supported by the Simons Foundation under the program ``Targeted Grants to Institutes”.

\section{Exact Categories}

Recall that if $\mathpzc{E}$ is an additive category then a \textbf{kernel-cokernel pair} in $\mathpzc{E}$ is a sequence
\begin{displaymath}
\xymatrix{
0\ar[r] & X\ar[r]^{f} & Y\ar[r]^{g} & Z\ar[r] &0
}
\end{displaymath}
where $g$ is a cokernel of $f$ and $f$ is a kernel of $g$. A \textbf{Quillen exact category} is a pair $(\mathpzc{E},\mathcal{Q})$ where $\mathpzc{E}$ is an additive category, and $\mathcal{Q}$ is a class of kernel-cokernel pairs in $\mathpzc{E}$ satisfying some axioms which make many of the constructions of homological algebra possible. For details one can consult \cite{Buehler}. In particular, there is a sensible definition of what it means for a map $f:X_{\bullet}\rightarrow Y_{\bullet}$ of complexes in $Ch(\mathpzc{E})$ to be a quasi-isomorphism. 

\begin{notation}
Let $(\mathpzc{E},\mathcal{Q})$ be an exact category.
\begin{enumerate}
\item
If $g$ appears as a cokernel in a short exact sequence in $\mathcal{Q}$, then $g$ is said to be an \textbf{admissible epimorphism}. The class of all admissible epimorphisms is denoted $\textbf{AdEpi}$.
\item
If $f$ appears as a kernel in a short exact sequence in $\mathcal{Q}$, then $f$ is said to be an \textbf{admissible monomorphism}. The class of all admissible monomorphisms is denoted $\textbf{AdMon}$.
\item
$\textbf{SplitMon}$ is the class of split monomorphisms, i.e. morphisms $i$ for which there existence a map $p$ such that $p\circ i$ is the identity. 
\end{enumerate}
\end{notation}

When it will not cause confusion, we will often denote an exact category by its underlying additive category $\mathpzc{E}$, with the class of exact sequences being understood.

\subsection{Derived functors in homotopical cateories}
Let $(\mathpzc{E},\mathcal{Q})$ be an exact category and denote by $\mathcal{W}_{\mathcal{Q}}$ the class of quasi-isomorphisms in $Ch(\mathpzc{E})$. $\mathcal{W}_{\mathcal{Q}}$ satisfies the $2$-out-of-$6$ property, and therefore $(Ch(\mathpzc{E}),\mathcal{W}_{\mathcal{Q}})$ is a homotopical category in the sense of \cite{Riehl}, as are the full subcategories $Ch_{+}(\mathpzc{E})$, $Ch_{\ge0}(\mathpzc{E})$, $Ch_{-}(\mathpzc{E})$, and $Ch_{\le0}(\mathpzc{E})$ of bounded below, non-negatively graded, bounded above, and non-positively graded complexes respectively.

\begin{defn}[\cite{Riehl} Definition 2.1.1]
A \textbf{homotopical category} is a pair $(\mathpzc{M},\mathcal{W})$ where $\mathpzc{M}$ is a category and $\mathcal{W}$ is a wide subcategory which satisfies the $2$-out-of-$6$ property: if $f:m\rightarrow m'$, $g:m'\rightarrow m''$ and $h:m''\rightarrow m'''$ are maps such that $h\circ g$ and $g\circ f$ are in $\mathcal{W}$, then $f,g,h$, and $h\circ g\circ f$ are all in $\mathcal{W}$. A functor between homotopical categories which preserves weak equivalences is said to be a \textbf{relative functor}.
\end{defn}

\begin{defn}[\cite{Riehl}, Definition 2.1.6]
The \textbf{homotopy category}, $\Ho(\mathpzc{M})$ of a homotopical category $(\mathpzc{M},\mathcal{W})$ is the formal localization of $\mathpzc{M}$ at $\mathcal{W}$.
\end{defn}

For $\mathpzc{M}$ a homotopical category we denote by $\gamma_{\mathpzc{M}}$ the canonical functor $\mathpzc{M}\rightarrow\Ho(\mathpzc{M})$.

\begin{defn}
Let $F:\mathpzc{M}\rightarrow\mathpzc{N}$ be a functor between homotopical categories.
\begin{enumerate}
\item
A\textbf{total left derived functor }$\textbf{L}F$ of $F$ is a left Kan extension $Lan_{\gamma_{\mathpzc{M}}}\gamma_{\mathpzc{N}}\circ F$. 
\item
A \textbf{left derived functor }$\mathbb{L}F$ of $F$ is a functor $\mathbb{L}F:\mathpzc{M}\rightarrow\mathpzc{N}$ such that $\gamma_{\mathpzc{N}}\circ\mathbb{L}F$ is a total left derived functor of $F$. 
\end{enumerate}
\end{defn}

Note that a total left derived functor is unique up to unique equivalence, but a left derived functor is by no means unique. 

\begin{defn}[\cite{Riehl} Definition 2.1.1, Definition 2.2.4]\label{def:homotopcialdeformation}
\begin{enumerate}
\item
Let $\mathpzc{M}$ be a homotopical category. A \textbf{left deformation on }$\mathpzc{M}$ is pair $(Q,\eta)$, where $Q:\mathpzc{M}\rightarrow\mathpzc{M}$ is an endomorphism, and $\eta:Q\rightarrow Id_{\mathpzc{M}}$ is a natural weak equivalence. 
\item
Let $F:\mathpzc{M}\rightarrow\mathpzc{N}$ be a functor between homotopical categories. A \textbf{left deformation for }$F$ is a left deformation $(Q,\eta)$ of $\mathpzc{M}$ such that $F$ preserves weak equivalences between objects in the essential image of $Q$. 
\end{enumerate}
\end{defn}

Dually one defines \textbf{right deformations} of homotopical categories and of functors.

\begin{thm}[\cite{Riehl} Theorem 2.2.8]
If $F:\mathpzc{M}\rightarrow\mathpzc{N}$ has a left deformation $(Q,\eta)$, then $F\circ Q$ is a left derived functor of $F$. 
\end{thm}

For $\mathpzc{E}$ an exact category we write $D(\mathpzc{E})\defeq\Ho(Ch(\mathpzc{E}))$. Similarly we define $D_{+}(\mathpzc{E})$, $D_{\ge0}(\mathpzc{E}), D_{-}(\mathpzc{E})$, and $D_{\le0}(\mathpzc{E})$. 

\subsection{Exactness of (Co)Limits}
To construct unbounded derived functors, we will need certain (co)limits to be exact. Let $\mathcal{I}$ be a  category, $\mathpzc{E}$ an exact category, and $\mathcal{S}$ a class of morphisms in $\mathpzc{E}$. Denote by $\mathpzc{Fun}_{\mathcal{S}}(\mathcal{I},\mathpzc{E})$ the category of functors $F:\mathcal{I}\rightarrow\mathpzc{E}$ such that for any map $i\rightarrow j$ in $\mathcal{I}$, $F(i)\rightarrow F(j)$ is in $\mathcal{S})$. 
A kernel-cokernel pair
$$0\rightarrow X\rightarrow Y\rightarrow Z\rightarrow 0$$
in $\mathpzc{Fun}_{\mathcal{S}}(\mathcal{I},\mathpzc{E})$ is said to be a \textbf{short exact sequence} if for each $i\in\mathcal{I}$
$$0\rightarrow X(i)\rightarrow Y(i)\rightarrow Z(i)\rightarrow 0$$
is an exact sequence in $\mathpzc{E}$. This does not necessarily endow $\mathpzc{Fun}_{\mathcal{S}}(\mathcal{I},\mathpzc{E})$ with the structure of an exact category, but we can still make sense of what it means for functors out of such categories to be exact, preserve acyclic complexes, etc. 

\begin{defn}
Let $\mathcal{S}$ be a class of morphisms in an exact category $\mathpzc{E}$.
\begin{enumerate}
\item
$\mathpzc{E}$ is said to be \textbf{weakly }$(\omega,\mathcal{S})$-\textbf{elementary} if the functor $\colim:\mathpzc{Fun}_{\mathcal{S}}(\omega,\mathpzc{E})\rightarrow\mathpzc{E}$ exists and is exact.
\item
$\mathpzc{E}$ is said to be \textbf{weakly }$(\omega,\mathcal{S})$-\textbf{coelementary} if the functor $\lim:\mathpzc{Fun}_{\mathcal{S}}(\omega^{op},\mathpzc{E})\rightarrow\mathpzc{E}$ exists and is exact.
\end{enumerate}
We will specifically be interested in weakly $(\omega,\textbf{AdMon})$-elementary exact categories, and weakly $(\omega,\textbf{AdEpi})$-coelementary exact categories
\end{defn}

\begin{example}
Any exact category which is countably cocomplete and has enough injectives is  weakly $(\omega,\textbf{AdMon})$-elementary, and any exact category which is countably complete and has enough projectives is weakly $(\omega,\textbf{AdEpi})$-coelementary c.f. Proposition \ref{prop:enoughprojml}.
\end{example}

It will be convenient to have the following stronger notion as well.

\begin{defn}[Following \cite{jensen254foncteurs} Page 6]
Let $\mathpzc{E}$ be an exact category and $\mathcal{I}$ a filtered category. 
\begin{enumerate}
\item
 An object $F\in\mathpzc{Fun}_{\textbf{AdMon}}(\mathcal{I},\mathpzc{E})$ is said to be \textbf{weakly co-flasque} if for any directed subset $J$ of $I$ the map $\colim_{j\in J}F_{j}\rightarrow\colim_{i\in I}F_{i}$ is an admissible monomorphism. The full subcategory of $\mathpzc{Fun}_{\textbf{AdMon}}(\mathcal{I},\mathpzc{E})$ consisting of weakly co-flasque objects is denoted $\mathpzc{Fun}^{wcfl}_{\textbf{AdMon}}(\mathcal{I},\mathpzc{E})$
 \item
  An object $F\in\mathpzc{Fun}_{\textbf{AdEpi}}(\mathcal{I}^{op},\mathpzc{E})$ is said to be \textbf{weakly flasque} if for any directed subset $J$ of $I$ the map $\lim_{i\in I}F_{i}\rightarrow\lim_{j\in J}F_{j}$ is an admissible epimorphism. The full subcategory of $\mathpzc{Fun}_{\textbf{AdEpi}}(\mathcal{I}^{op},\mathpzc{E})$ consisting of weakly flasque objects is denoted $\mathpzc{Fun}^{wfl}_{\textbf{AdEpi}}(\mathcal{I}^{op},\mathpzc{E})$
 \end{enumerate}
\end{defn}

\begin{rem}
If $F\in \mathpzc{Fun}^{wcfl}_{\textbf{AdMon}}(\mathcal{I},\mathpzc{E})$, and $D:\mathcal{J}\rightarrow\mathcal{I}$ is a cofinal functor, then $F\circ D\in \mathpzc{Fun}^{wcfl}_{\textbf{AdMon}}(\mathcal{J},\mathpzc{E})$, and of course dually for $F\in \mathpzc{Fun}^{wfl}_{\textbf{AdEpi}}(\mathcal{J}^{op},\mathpzc{E})$.
\end{rem}

\begin{prop}
Let $F\in\mathpzc{Fun}_{\textbf{AdMon}}(\mathcal{I},\mathpzc{E})$ be weakly co-flasque, and let $\mathcal{J}\subset\mathcal{I}$ be a direct subcategory. Then $F|_{\mathcal{J}}$ is in $\mathpzc{Fun}_{\textbf{AdMon}}^{wcfl}(\mathcal{J},\mathpzc{E})$. 
\end{prop}

\begin{proof}
Let $J'\subset J$ be directed. We know that $\colim_{\mathcal{J'}}F|_{J'}\rightarrow\colim_{\mathcal{I}}F$ is an admissible monomorphism. But this map factors through $\colim_{\mathcal{J'}}F|_{J'}\rightarrow \colim_{\mathcal{J}}F|_{J}$. By the obscure axiom this map must also be an admissible monomorphism.
\end{proof}

\begin{defn}\label{defn:ML}
Let $\mathpzc{E}$ be an exact category. 
\begin{enumerate}
\item
$\mathpzc{E}$ is said to be \textbf{weakly co-ML} if for any filtered category $\mathcal{I}$ the functor $\colim:Ch_{\ge0}(\mathpzc{Fun}^{wcfl}_{\textbf{AdMon}}(\mathcal{I},\mathpzc{E}))\rightarrow Ch_{\ge0}(\mathpzc{E})$ exists and sends term-wise acyclic complexes to acyclic complexes. It is said to be \textbf{co-ML} if in addition it has all filtered colimits, and \textbf{strongly co-ML} if it is co-ML and the functor $\colim:\mathpzc{Fun}_{\textbf{AdMon}}(\mathcal{I},\mathpzc{E})\rightarrow\mathpzc{E}$ is exact for any filtered category $\mathcal{I}$.
\item
$\mathpzc{E}$ is said to be \textbf{weakly co-ML} if for any filtered category $\mathcal{I}$ the functor $\lim:Ch_{\le0}(\mathpzc{Fun}^{wfl}_{\textbf{AdEpi}}(\mathcal{I}^{op},\mathpzc{E}))\rightarrow Ch_{\le0}(\mathpzc{E})$ exists and sends term-wise acyclic complexes to acyclic complexes.   It is said to be \textbf{ML} if in addition it has all projective limits, and \textbf{strongly ML} if it is ML and the functor $\lim:\mathpzc{Fun}_{\textbf{AdEpi}}(\mathcal{I}^{op},\mathpzc{E})\rightarrow\mathpzc{E}$ is exact for any filtered category $\mathcal{I}$.
\end{enumerate}
\end{defn}

ML here stands for `Mittag-Leffler'. 

\begin{example}
The category $\mathpzc{Ab}$ of abelian groups is  ML and strongly co-ML. 
\end{example}

\begin{proof}
Filtered colimits are exact in $\mathpzc{Ab}$, so that is is strongly co-ML is clear. To show that it is ML, it suffices to show that for $F\in \mathpzc{Fun}^{wfl}_{\textbf{AdEpi}}(\mathcal{I}^{op},\mathpzc{Ab})$ the map $\mathbb{L}\lim_{\mathcal{I}^{op}}F(i)\rightarrow\lim_{\mathcal{I}^{op}}F(i)$ is an equivalence. By picking a cofinal functor from a directed category, we may assume that $\mathcal{I}$ is directed. This is then Proposition 1.6 in \cite{jensen254foncteurs}.
\end{proof}

\begin{prop}\label{prop:enoughprojml}
Let $\mathpzc{E}$ be an exact category.
\begin{enumerate}
\item
If $\mathpzc{E}$ is has enough injectives then for any filtered set $\mathcal{I}$, the functor $\colim:Ch_{\ge0}(\mathpzc{Fun}^{wcfl}_{\textbf{AdMon}}(\mathcal{I},\mathpzc{E}))\rightarrow Ch_{\ge0}(\mathpzc{E})$ sends acyclic complexes to acyclic complexes.
\item
If $\mathpzc{E}$ is complete and has enough projectives then for any filtered set $\mathcal{I}$, the functor $\lim:Ch_{\le0}(\mathpzc{Fun}^{wfl}_{\textbf{AdEpi}}(\mathcal{I}^{op},\mathpzc{E}))\rightarrow Ch_{\le0}(\mathpzc{E})$  sends acyclic complexes to acyclic complexes.
\end{enumerate}
\end{prop}

\begin{proof}
We prove the second claim, the first being similar. Let $X$ be an acyclic complex in
$Ch_{\le0}(\mathpzc{Fun}^{wfl}_{\textbf{AdEpi}}(\mathcal{I}^{op},\mathpzc{E}))$. We need to show that $\lim_{\mathcal{I}^{op}} X_{i}$ is acyclic in $Ch_{\le0}(\mathpzc{E})$.
 Since $\mathpzc{E}$ has enough projectives, it suffices to show that $Hom(P,\lim_{\mathcal{I}^{op}} X_{i})\cong\lim_{\mathcal{I}^{op}}Hom(P,X_{i})$
is acyclic in $Ch_{\le0}(\mathpzc{Ab})$ for any projective $P$. But since $P$ is projective $i\mapsto Hom(P,X_{i})$ is in $Ch_{\le0}(\mathpzc{Fun}^{wfl}_{\textbf{AdEpi}}(\mathcal{I}^{op},\mathpzc{Ab}))$, so this follows from the fact that $\mathpzc{Ab}$ is ML. 
\end{proof}

\begin{cor}
Any exact category which has filtered colimits and has enough injectives is  co-ML, and any exact category which has projective limits and has enough projectives is ML.
\end{cor}

\section{Embedding Theorems}

 In this section we show that any (small) weakly $(\omega,\textbf{AdEpi})$-coelementary (resp. $(\omega,\textbf{AdMon})$-elementary) exact category embeds in a weakly $(\omega,\textbf{AdEpi})$-coelementary (resp. $(\omega,\textbf{AdMon})$-elementary) abelian categoriy.  The first point to note is that any small exact category $\mathpzc{E}$ has both a left and right abelianization.

\begin{defn}
We call an fully faithful functor $I:\mathpzc{E}\rightarrow\mathpzc{A}$ from an exact category to an abelian category  a \textbf{left abelianization} of $\mathpzc{E}$ if 
\begin{enumerate}
\item
$I$ is fully faithful.
\item
$I$ is exact.
\item
$I$ reflects exactness.
\item
The essential image of $I$ is closed under extensions.
\item
$I$ preserves all kernels which exist.
\item
If $f$ is a morphism in $\mathpzc{E}$, then $f$ is an admissible epic if and only if $I(f)$ is an epic.
\end{enumerate}
Dually one defines a \textbf{right abelianization}.
\end{defn}

\begin{rem}
This terminology comes from \cite{kelly2016homotopy}, and is inspired by the left heart of a quasi-abelian category. However the right abelian envelope of an exact category as discussed in \cite{bodzenta2020abelian} is a \textit{left abelianization}.
\end{rem}

Abelianizations give us a reasonable notion of homology which can detect quasi-isomorphisms.

\begin{notation}
Let $J:\mathpzc{E}\rightarrow\mathpzc{A}$ be an exact functor from an exact category to an abelian category. For $X_{\bullet}\in Ch(\mathpzc{E})$ and $n\in\mathbb{Z}$ we write
$$H^{J}_{n}(X_{\bullet})\defeq H_{n}(J(X_{\bullet}))$$
\end{notation}

Appendix A in \cite{Buehler} guarantees that any small exact category has both a left abelianziation $I^{st}_{l}:\mathpzc{E}\rightarrow\mathpzc{A}^{st}_{l}$, and a right abelianization $I^{st}_{r}:\mathpzc{E}\rightarrow\mathpzc{A}^{st}_{r}$. The left abelianization constructed in \cite{Buehler} is the category of sheaves on a site $(\mathpzc{E},\tau)$, where $\tau$ is the Grothendieck topology on $\mathpzc{E}$ whose covers are of the form $\{p:C\rightarrow D\}$, where $p$ is an admissible epimorphism and $I^{st}_{l}$ is the Yoneda embedding. In particular the functor $I_{l}$ commutes with  limits. 
%
%

\begin{cor}\label{prop:truncquasi}
Let $\mathpzc{E}$ be an exact category with kernels. Then the functors $\tau^{L}_{\ge n}$ preserve quasi-isomorphisms. 
\end{cor}

\begin{proof}
Pick a left abelianziation $I_{l}:\mathpzc{E}\rightarrow\mathpzc{A}$. Then there is a natural isomorphism $I_{l}\circ \tau^{L}_{\ge n}\cong \tau^{L}_{\ge n}\circ I_{l}$. Since truncations preserve weak equivalences in abelian categories, and $I_{l}$ reflects equivalences, this proves the claim.
\end{proof}
 
\begin{defn}
An $\omega$-\textbf{left abelianization} of an exact category $\mathpzc{E}$, is a left abelianiziation $I_{l}:\mathpzc{E}\rightarrow\mathpzc{A}_{l}$ such that
\begin{enumerate}
\item
Filtered colimits are exact in $\mathpzc{A}_{l}$.
\item
$\mathpzc{A}_{l}$ is weakly $(\omega,\textbf{AdEpi})$-coelementary. 
\item
$I_{l}$ commutes with limits.
\end{enumerate}
Dually one defines a $\omega$-\textbf{right abelianization}.
\end{defn}

The main technical result is that that small weakly $(\omega,\textbf{AdEpi})$-coelementary exact categories have $\omega$-left abelianizations and, dually, small weakly $(\omega,\textbf{AdMon})$-elementary exact categories have $\omega$-right abelianizations.

\begin{lem}\label{lem:filteredthingsheaves}
Let $\mathpzc{E}$ be a small exact category, and 
\begin{displaymath}
\xymatrix{
0\ar[r] & K\ar[r]^{\alpha} & G\ar[r]^{\beta} & H\ar[r] & 0
}
\end{displaymath}
be an exact sequence of functors $\omega^{op}\rightarrow\mathpzc{A}^{st}_{l}$. Let $D$ be an object of $\mathpzc{E}$ and $y$ an element of $\lim_{\omega^{op}}G_{n}(D)$. Denote by $\pi_{m}^{G}:\lim_{\omega^{op}}G_{m}\rightarrow G_{m}$ and $\pi_{m}^{H}:\lim_{\omega^{op}}H_{m}\rightarrow H_{m}$ the projection maps. Then 
\begin{enumerate}
\item
There is a functor $\overline{C}_{-}\in\mathpzc{Fun}_{\textbf{AdEpi}}(\omega^{op},\mathpzc{E})$ and compatible epimorphisms $\overline{\pi}_{n}:\overline{C}_{n}\rightarrow D$
\item
There are elements $\overline{x}_{n}\in G_{n}(\overline{C}_{n})$
\end{enumerate}
such that 
\begin{enumerate}
\item
Denoting by $g_{n}$ the map $G_{n}\rightarrow G_{n-1}$, and by $\overline{f}_{n}$ the map $\overline{C}_{n}\rightarrow\overline{C}_{n-1}$, for each $n$, $g_{n}^{\overline{C}_{n}}(\overline{x}_{n})=G_{n-1}(\overline{f}_{n})(\overline{x}_{n-1})$.
\item
For each $n$ $\beta_{n}^{\overline{C}_{n}}(\overline{x}_{n})=H_{i}(\overline{\pi_{n}})(\pi^{H}_{n}(y))$.
\end{enumerate}
\end{lem}

\begin{proof}
Since $\beta_{0}:G_{0}\rightarrow H_{0}$ is an epimorphism of sheaves, there is an admissible epimorphism $\overline{\pi}_{0}:\overline{C}_{0}\rightarrow D$ and $\overline{x}_{0}\in G_{n}(\overline{C}_{0})$ such that $\beta_{0}^{\overline{C}_{0}}(\overline{x}_{0})=H_{0}(\overline{\pi}_{0})(\pi^{H}_{0}(y))$. Suppose we have defined $\overline{C}_{n}$ and $\overline{x}_{n}$. Consider the pullback diagram
\begin{displaymath}
\xymatrix{
G_{n}\ar[r]^{\beta_{n}} & H_{n}\\
G_{n}\times_{H_{n}}H_{n+1}\ar[u]\ar[r]^{\beta_{n+1}} & H_{n+1}\ar[u]
}
\end{displaymath}
in which all maps are epimorphisms. $(\overline{x}_{n},H_{n+1}(\overline{\pi}_{n})(\pi^{H}_{n+1}(y)))$ is a well defined element of $(G_{n}\times_{H_{n}}H_{n+1})(\overline{C}_{n})$. Now we have a commutative diagram of exact sequences
\begin{displaymath}
\xymatrix{
0\ar[r] & K_{n+1}\ar[r]\ar[d] & G_{n+1}\ar[r]\ar[d] & H_{n+1}\ar[r]\ar[d] & 0\\
0\ar[r] & K_{n}\ar[r] & G_{n}\times_{H_{n}}H_{n+1}\ar[r] & H_{n}\ar[r] & 0
}
\end{displaymath} 
The left and right vertical maps are epimorphisms, so the middle one must be as well. Thus there is some admissible epimorphism $\overline{f}_{n+1}:\overline{C}_{n+1}\rightarrow\overline{C}_{n}$ and some $\overline{x}_{n+1}\in G_{n+1}(\overline{C}_{n+1})$ such that
$$(g^{\overline{C}_{n+1}}_{n+1}(\overline{x_{n+1}}),\beta^{\overline{C}_{n+1}}_{n+1}(\overline{x_{n+1}}))=(G_{n}(\overline{f}_{n+1})(\overline{x}_{n}),H_{n+1}(\overline{\pi}_{n}\circ\overline{f}_{n+1})(\pi^{H}_{n+1}(y)))$$
Now we just set $\overline{\pi}_{n+1}\defeq\overline{\pi}_{n}\circ\overline{f}_{n+1}$
\end{proof}

\begin{cor}
Let $\mathpzc{E}$ be a weakly $(\omega,\textbf{AdEpi})$-coelementary exact category.  Then the abelian category $\mathpzc{A}^{st}_{l}(\mathpzc{E})$ is weakly $(\omega,\textbf{AdEpi})$-coelementary 
\end{cor}

\begin{proof}
Let $\mathpzc{E}$ be an exact category, and 
\begin{displaymath}
\xymatrix{
0\ar[r] & K\ar[r]^{\alpha} & G\ar[r]^{\beta} & H\ar[r] & 0
}
\end{displaymath}
be an exact sequence of functors $\omega^{op}\rightarrow\mathpzc{A}^{st}_{l}$. It suffices to show that $\lim_{\omega^{op}}\beta_{n}:\lim_{\omega^{op}}G_{n}\rightarrow \lim_{\omega^{op}}H_{n}$ is an epimorphism. Let $D$ be an object of $\mathpzc{E}$, and let $y\in  \lim_{\omega^{op}}H_{n}(D)$. Denote by $\pi_{m}^{G}:\lim_{\omega^{op}}G_{m}\rightarrow G_{m}$ and $\pi_{m}^{H}:\lim_{\omega^{op}}H_{m}\rightarrow H_{m}$ the projection maps.  Let $\overline{C}_{-}\in\mathpzc{Fun}_{\textbf{AdEpi}}(\omega^{op},\mathpzc{E})$, $\overline{x}_{n}$, and $\overline{\pi}_{n}$ be as in the statement of Lemma \ref{lem:filteredthingsheaves}. Write $\overline{C}\defeq\lim_{\omega^{op}}\overline{C}_{n}$, and let $\overline{\pi}_{\infty}:\overline{C}\rightarrow D$ be the projection. Since $\mathpzc{E}$ is weakly $(\omega,\textbf{AdEpi})$-coelementary the map $\overline{\pi}_{\infty}$ is an admissible epimorphism by \cite{kelly2016homotopy} Proposition 2.96, as are the maps $\overline{p}_{n}:\overline{C}\rightarrow\overline{C}_{n}$. The sequence $x=(G_{n}(\overline{p}_{n})(\overline{x}_{n}))$ is a well-defined element of $\lim_{\omega^{op}}G_{n}(\overline{C})$, and $\lim_{\omega^{op}}\beta^{\overline{C}}(x)=(\lim_{\omega^{op}}G_{n}(\overline{\pi}_{\infty}))(y)$. This shows surjectivity of $\lim_{\omega^{op}}\beta$. 
\end{proof}

\begin{cor}
\begin{enumerate}
\item
Every small weakly idempotent complete, weakly $(\omega,\textbf{AdEpi})$-coelementary exact category has a $\omega$-left abelianization.
\item
Every small weakly idempotent complete, weakly $(\omega,\textbf{AdMon})$-elementary  exact category has a  $\omega$-right abelianization. 
\end{enumerate}
\end{cor}

%
%

\subsection{Aside: Quasi-abelian Categories}

Many of the examples of exact categories we consider are quasi-abelian categories: 

\begin{defn}
A \textbf{quasi-abelian category} is a finitely complete and finitely cocomplete additive category $\mathpzc{E}$ such that the class of all kernel-cokernel pairs form a Quillen exact structure on $\mathpzc{E}$. 
\end{defn}
We regard a quasi-abelian category $\mathpzc{E}$ as being equipped with a canonical Quillen exact structure given by the class $\mathcal{Q}_{can}$ of all kernel-cokernel pairs.  

If $\mathpzc{E}$ is a quasi-abelian category then there are two canonical $t$-structures on the derived category called the \textbf{left} and \textbf{right} $t$-structures respectively. The respective hearts are called the \textbf{left heart}, denoted $LH(\mathpzc{E})$, and the \textbf{right heart}, denoted $RH(\mathpzc{E})$. Section 1.2 of \cite{qacs} shows in particular that there are canonical embeddings $\mathpzc{E}\rightarrow LH(\mathpzc{E})$ and $\mathpzc{E}\rightarrow RH(\mathpzc{E})$ such that the former is a left abelianization and the latter a right abelianization. Moreover if $\mathpzc{E}$ is weakly $(\omega,\textbf{AdEpi})$-coelementary, then a similar proof to Corollary 1.4.7 in \cite{qacs} can be given to show that in this case the left heart is an $\omega$-left abelianization, and dually, the right heart is an $\omega$-right abelianization.

\subsection{Deriving $\lim_{\omega^{op}}$}
Before moving onto the general theory of deriving unbounded functors in the next section, we need some results on deriving the projective limit functor $\lim_{\omega^{op}}:\mathpzc{Fun}(\omega^{op},\mathpzc{E})\rightarrow\mathpzc{E}$. We will prove these results by passing to abeliainizations, although we expect it is reasonably easy to prove them internally to exact categories. Note that although we technically need $\mathpzc{E}$ to be small to pass to abelianizations, since we are working with small diagrams \cite{kelly2016homotopy} Section 1.2, for the proofs in this section it is enough that $\mathpzc{E}$ is locally small.

Recall (e.g. \cite{stacks-project} Definition 13.34.1) that if $\mathcal{T}$ is a triangulated category with countable products, and $K\in\mathpzc{Fun}(\omega^{op},\mathcal{T})$ is a functor, then the \textbf{derived limit of }$K$ is given by the distinguished triangle
$$R\lim_{\omega^{op}}K\rightarrow\prod_{n} K^{n}\rightarrow\prod_{n} K^{n}\rightarrow R\lim_{\omega^{op}}K[-1]$$
with $\prod_{n} K_{n}\rightarrow\prod_{n} K_{n}$ being the is the map $Id-k$, and $k$ is induced by the maps $k^{n}:K^{n}\rightarrow K^{n-1}$. In particular if $\eta:K\rightarrow L$ is a natural transformation in $\mathpzc{Fun}(\omega^{op},\mathcal{D})$ then there is an induced map $R\lim_{\omega^{op}}\eta:R\lim_{\omega^{op}}K\rightarrow R\lim_{\omega^{op}}L$. If $\psi$ is a natural isomorphism then so is $R\lim_{\omega^{op}}\psi$. 

\begin{example}
Let $\mathpzc{E}$ be an exact category with exact countable products. Then as in \cite{stacks-project} Lemma 13.34.2, the triangulated category $\Ho(Ch(\mathpzc{E}))$ has countable products. On objects they are computed by taking term-wise products of complexes.
\end{example}

The result below can be proven as in \cite{stacks-project} Lemma 15.85.1 and Lemma 15.85.6, where we also use Theorem 5.2.4 in \cite{prosmans1999derived}.

\begin{prop}
Let $\mathpzc{A}$ be an abelian category with exact countable products
\begin{enumerate}
\item
$\lim_{\omega^{op}}:Ch(\mathpzc{Fun}(\omega^{op},\mathpzc{A}))\rightarrow Ch(\mathpzc{A})$ is right deformable. 
\item
If $X^{n}_{\bullet}$ is a complex in $Ch(\mathpzc{Fun}(\omega^{op},\mathpzc{A}))$ such that for each $m$ $X^{n}_{m}$ is $\lim_{\omega^{op}}$-acyclic, then the map $\lim_{\omega^{op}}X^{n}_{\bullet}\rightarrow \mathbb{R}\lim_{\omega^{op}}X^{n}_{\bullet}$ is an equivalence.
\item
 $\mathbb{R}\lim_{\omega^{op}}X^{n}_{\bullet}\cong R\lim_{\omega^{op}}X^{n}_{\bullet}$
 \end{enumerate}
\end{prop}

\begin{cor}
 If $\mathpzc{A}$ is a weakly $(\omega,\textbf{AdEpi})$-coelementary abelian category, and $K\in Ch(\mathpzc{Fun}(\omega^{op},\mathpzc{A}))$ is a functor such that for each $n$, $K^{n+1}\rightarrow K^{n}$ is an epimorphism in each degree, then $\lim_{\omega^{op}}X^{n}_{\bullet}\rightarrow \mathbb{R}\lim_{\omega^{op}}X^{n}_{\bullet}$ is an equivalence.
\end{cor}

\begin{prop}\label{prop:limithomequiv}
Let $\mathpzc{A}$ be an abelian category with countable exact products. Let $K,L\in\mathpzc{Fun}(\omega^{op},Ch(\mathpzc{A}))$ be $\lim_{\omega^{op}}$-acyclic diagrams, and let $f:\lim_{\omega^{op}}K^{n}\rightarrow\lim_{\omega^{op}}L^{n}$ be a map. Suppose there is a natural equivalence $\psi:K\rightarrow L$ in $\mathpzc{Fun}(\omega^{op},\Ho(Ch(\mathpzc{A})))$ such that the diagram 
\begin{displaymath}
\xymatrix{
\lim_{\omega^{op}}K^{n}\ar[d]^{f}\ar[r] & K^{n}\ar[d]^{\psi^{n}}\\
\lim_{\omega^{op}}L^{n}\ar[r] & L^{n}
}
\end{displaymath}
commutes in $\Ho(Ch(\mathpzc{A}))$. Then $f$ is an equivalence. 
\end{prop}

\begin{proof}
We have commutative diagrams
\begin{displaymath}
\xymatrix{
\lim_{\omega^{op}}K^{n}\ar[d]^{f}\ar[r] & K^{n}\ar[d]^{\psi^{n}}\\
\lim_{\omega^{op}}L^{n}\ar[r] & L^{n}
}
\end{displaymath}
and therefore a commutative diagram
\begin{displaymath}
\xymatrix{
\lim_{\omega^{op}}K^{n}\ar[d]^{f}\ar[r] & \prod _{n}K^{n}\ar[d]^{\psi^{n}}\\
\lim_{\omega^{op}}L^{n}\ar[r] & \prod_{n}L^{n}
}
\end{displaymath}
Thus we get a commutative diagram
\begin{displaymath}
\xymatrix{
\lim_{\omega^{op}}K^{n}\ar[d]^{f}\ar[r] & \prod _{n}K^{n}\ar[d]^{\psi^{n}}\ar[r] & \prod _{n}K^{n}\ar[d]^{\psi^{n}}\ar[r] & \lim_{\omega^{op}}K^{n}[-1]\ar[d]\\
\lim_{\omega^{op}}L^{n}\ar[r] & \prod_{n}L^{n}\ar[r] & \prod _{n}K^{n}\ar[r] & \lim_{\omega^{op}}K^{n}[-1]
}
\end{displaymath}
Since $K$ and $L$ are $\lim_{\omega^{op}}$-acyclic, this is a distinguished triangle. The middle vertical maps are equivalences, so the left vertical map is also an equivalence. 
\end{proof}

\begin{cor}\label{cor:limithomequiv}
Let $\mathpzc{E}$ be a weakly $(\omega,\textbf{AdEpi})$-coelementary exact category. Let $K,L\in\mathpzc{Fun}(\omega^{op},Ch(\mathpzc{E}))$ be $\lim_{\omega^{op}}$-acyclic diagrams and $f:\lim_{\omega^{op}}K^{n}\rightarrow\lim_{\omega^{op}}L^{n}$ be a map. Suppose there is a natural equivalence $\psi:K\rightarrow L$ in $\mathpzc{Fun}(\omega^{op},\Ho(Ch(\mathpzc{E})))$ such that the diagram below commutes
\begin{displaymath}
\xymatrix{
\lim_{\omega^{op}}K^{n}\ar[d]^{f}\ar[r] & K^{n}\ar[d]^{\psi^{n}}\\
\lim_{\omega^{op}}L^{n}\ar[r] & L^{n}
}
\end{displaymath}
commutes in $\Ho(Ch(\mathpzc{E}))$. Then $f$ is an equivalence. 
\end{cor}

\begin{proof}
By picking an $\omega$-right abelianization we reduce to the situation of Proposition \ref{prop:limithomequiv}.
\end{proof}

\section{Deriving Unbounded Functors}

Let $\mathpzc{E}$ be an exact category. As mentioned earlier, for $*\in\{\emptyset,\ge0,\le0,+,-\}$ we can regard the categories $Ch_{*}(\mathpzc{E})$ as homotopical categories, where $\mathcal{W}$ is the class of quasi-isomorphisms. In this section we will be interested in deriving functors of the form
$$F:Ch_{*}(\mathpzc{D})\rightarrow Ch_{*}(\mathpzc{E})$$
where $\mathpzc{D}$ and $\mathpzc{E}$ are exact categories, and $F$ arises from an additive functor $F:\mathpzc{D}\rightarrow\mathpzc{E}$. Note we are abusing notation here by also denoting by $F$ the functor at the level of complexes. This is to simplify notation throughout. 

\subsection{Deformation Functors for Exact Categories}

Let $\mathpzc{E}$ be an exact category, and $i_{\mathcal{B}}:\mathcal{B}\rightarrow Ch(\mathpzc{D})$  a full subcategory. $\mathcal{B}$ is said to be \textbf{degree-wise split extension closed} if for any sequence
$$0\rightarrow X\rightarrow Y\rightarrow Z\rightarrow 0$$
in $Ch(\mathpzc{E})$ which is split exact and each degree, and which has $X,Z\in\mathcal{B}$, then $Y\in\mathcal{B}$. 

\begin{defn}
Let $\mathpzc{D}$ be an exact  category. A degree-wise split extension closed full subcategory $i_{\mathcal{B}}:\mathcal{B}\rightarrow Ch_{+}(\mathpzc{D})$ is said to be a \textbf{homological left deformation of} $Ch_{+}(\mathpzc{D})$ if there is a pair $(Q_{+},\eta)$ where $Q_{+}:Ch_{+}(\mathpzc{E})\rightarrow\mathcal{B}$ is a functor, and $\eta:i_{\mathcal{B}}\circ Q\rightarrow Id_{Ch_{+}(\mathpzc{E})}$ is a natural quasi-isomorphism, and admissible epimorphism in each degree such that whenever $X_{\bullet}$ is concentrated in degrees $\ge n$, $Q_{+}(X_{\bullet})$ is concentrated in degrees $\ge n$. 
\end{defn}

\begin{defn}
Let $\mathpzc{D}$ be an exact category.  A degree-wise split extension closed full subcategory $i_{\mathcal{B}}:\mathcal{B}\rightarrow Ch_{-}(\mathpzc{D})$ is said to be a \textbf{homological right deformation of } $Ch_{-}(\mathpzc{D})$ if there is a pair $(Q_{-},\eta)$ where $Q_{-}:Ch_{-}(\mathpzc{E})\rightarrow\mathcal{B}$ is a functor, and $\eta:Id_{Ch_{+}(\mathpzc{E})}\rightarrow i_{\mathcal{B}}\circ Q $ is a natural quasi-isomorphism, and admissible monomorphism in each degree such that whenever $X_{\bullet}$ is concentrated in degrees $\le n$, $Q_{+}(X_{\bullet})$ is concentrated in degrees $\le n$. 
\end{defn}

\begin{defn}
Let $F:\mathpzc{D}\rightarrow\mathpzc{E}$ be a functor.
\begin{enumerate}
\item
$i_{\mathcal{B}}:\mathcal{B}\rightarrow Ch_{+}(\mathpzc{D})$ is said to be a \textbf{homological left deformation of} $F$ if $Ch_{+}(\mathpzc{D})$ is a homological left deformation of $Ch_{+}(\mathpzc{D})$, and $F$ preserves weak equivalence between complexes in $\mathcal{B}$. 
\item
$i_{\mathcal{B}}:\mathcal{B}\rightarrow Ch_{-}(\mathpzc{D})$ is said to be a \textbf{homological right deformation of} $F$ if $Ch_{-}(\mathpzc{D})$ is a homological right deformation of $Ch_{-}(\mathpzc{D})$, and $F$ preserves weak equivalence between complexes in $\mathcal{B}$. 
\end{enumerate}
\end{defn}

Note that if $\mathcal{B}$ is a homological left (resp. right) deformation of $Ch_{+}(\mathpzc{D})$ (or of a functor $F$) then $(i_{\mathcal{B}}\circ Q_{+},\eta_{+})$ is a left (resp. right) deformation of $Ch_{+}(\mathpzc{D})$ (or $F$) in the sense of homotopical categories, i.e. Definition \ref{def:homotopcialdeformation}.

\begin{defn}
Let $\mathpzc{D}$ be an exact category, and $i_{\mathcal{S}}:\mathcal{S}\rightarrow\mathpzc{D}$ a full subcategory, with fully faithful inclusion functor $i_{\mathcal{S}}$. A $\mathcal{S}$-\textbf{left deformation functor} is a pair $(Q,\eta)$, where $Q:\mathpzc{D}\rightarrow\mathcal{S}$ is a functor, and $\eta:i_{\mathcal{S}}\circ Q\rightarrow i_{\mathpzc{D}}$ is a natural transformation which is an admissible epimorphism. One defines a $\mathcal{S}$-right deformation functor dually. 
\end{defn}

The following can be proven exactly as in \cite{Buehler} Theorem 12.7.

\begin{prop}
Let $\mathpzc{D}$ be an exact category, $i_{\mathcal{S}}:\mathcal{S}\rightarrow\mathpzc{D}$ a full subcategory closed under finite sums, and $(Q,\eta)$ a $\mathcal{S}$-left deformation functor. Then there is a functor $Q_{+}:Ch_{+}(\mathpzc{D})\rightarrow Ch_{+}(\mathcal{S})$, and a natural transformation $\eta_{+}:i_{\mathcal{S}}\circ Q_{+}\rightarrow Id_{Ch_{+}(\mathpzc{D})}$ such that $(Q_{+},\eta_{+})$ realises $Ch_{+}(\mathcal{S})$ as a left deformation of $Ch_{+}(\mathpzc{D})$. 
\end{prop}

Dually we have the following,

\begin{prop}
Let $\mathpzc{D}$ be an exact category, $i_{\mathcal{S}}:\mathcal{S}\rightarrow\mathpzc{D}$ a full subcategory closed under finite sums, and $(Q,\eta)$ a $\mathcal{S}$-right deformation functor. Then there is a functor $Q_{-}:Ch_{-}(\mathpzc{D})\rightarrow Ch_{-}(\mathcal{S})$, and a natural transformation $\eta_{-}:Id_{Ch_{-}(\mathpzc{D})}\rightarrow i_{\mathcal{S}}\circ Q_{-}$ such that $(Q_{-},\eta_{-})$ realises $Ch_{-}(\mathcal{S})$ as a right deformation of $Ch_{-}(\mathpzc{D})$. 
\end{prop}

\begin{defn}\label{defn:F-proj}
Let $F:\mathpzc{D}\rightarrow\mathpzc{E}$ be a functor of exact categories. A full subcategory $\mathcal{S}\subset\mathpzc{D}$ is said to be \textbf{left functorially adapted to }$F$ if 
\begin{enumerate}
\item
There is a $\mathcal{S}$-left deformation functor $(Q,\eta)$,
\item
$\mathcal{S}$ is closed under finite direct sums.
\item
If $0\rightarrow X\rightarrow Y\rightarrow Z\rightarrow 0$ is an exact sequence with $Y$ and $Z$ in $\mathcal{S}$, then $X$ in $\mathcal{S}$ and 
$$0\rightarrow F(X)\rightarrow F(Y)\rightarrow F(Z)\rightarrow 0$$
is an exact sequence in $\mathpzc{E}$. 
\end{enumerate}
\end{defn}
In the context of quasi-abelian categories, \cite{qacs} Definition 1.3.2 $\mathcal{S}$ is called an $F$-\textbf{projective subcategory} - though functoriality is not assumed there. 
We then have the following.

\begin{prop}
Let $F:\mathpzc{D}\rightarrow\mathpzc{E}$ be an additive functor of exact, and $\mathcal{S}\subset\mathpzc{D}$ a subcategory which is left functorially adapted to $F$. If $f:X\rightarrow Y$ is a quasi-isomorphism of complexes in $Ch_{+}(\mathcal{S})$, then $F(f)$ is a weak equivalence in $Ch_{+}(\mathpzc{E})$. In particular $Ch_{+}(\mathcal{S})$ is a homological left deformation of $F$.
\end{prop}

\begin{proof}
Let $f:X_{\bullet}\rightarrow Y_{\bullet}$ be a quasi-isomorphism of complexes in $Ch_{+}(\mathcal{S})$. Then $cone(f)$ is an acyclic complex in $Ch_{+}(\mathcal{S})$. It follows from Definition \ref{defn:F-proj} (3) that $F(cone(f))\cong (cone(F(f)))$ is acyclic, i.e. $F(f)$ is a weak equivalence.
\end{proof}

There is a dual situation.

\begin{defn}
Let $F:\mathpzc{D}\rightarrow\mathpzc{E}$ be a functor of quasi-abelian categories. A full subcategory $\mathcal{S}\subset\mathpzc{D}$ is said to be \textbf{right functorially adapted to }$F$ if 
\begin{enumerate}
\item
There is a full subcategory $\mathcal{S}$ of $\mathpzc{D}$, and an $\mathcal{S}$-right deformation functor $(Q,\eta)$,
\item
If $0\rightarrow X\rightarrow Y\rightarrow Z\rightarrow 0$ is an exact sequence with $X$ and $Y$ in $\mathcal{S}$, then $Z\in\mathcal{S}$ and
$$0\rightarrow F(X)\rightarrow F(Y)\rightarrow F(Z)\rightarrow 0$$
is an exact sequence in $\mathcal{S}$. 
\end{enumerate}
\end{defn}
Again in \cite{qacs}, this is called an $F$-injective subcategory (without functoriality). 

\begin{prop}
Let $F:\mathpzc{D}\rightarrow\mathpzc{E}$ be a functor of exact categories, and $\mathcal{S}\subset\mathpzc{D}$ a subcategory which is right functorially adapted to $F$. If $f:X\rightarrow Y$ is a weak equivalence of complexes in $Ch_{-}(\mathcal{S})$, then $F(f)$ is a weak equivalence in $Ch_{-}(\mathpzc{E})$. In particular $Ch_{-}(\mathcal{S})$ is a homological right deformation of $F$.
\end{prop}

\begin{defn}[\cite{qacs} Definition 1.3.6]
An additive functor $F:\mathpzc{D}\rightarrow\mathpzc{E}$ between exact categories is said to be \textbf{functorially explicitly right (resp. left) derivable} it $\mathpzc{E}$ has a subcategory which is left-adapted (resp. right-adapted) to $F$.
\end{defn}

\subsection{Left truncation systems}

To derive functors of unbounded complexes, we need to be able to construct our complexes from bounded below complexes in a controlled way.

\begin{defn}
A \textbf{left truncation system} on an exact category $\mathpzc{E}$ is a sequence
$$T_{0}\rightarrow T_{1}\rightarrow\ldots \rightarrow T_{n}\rightarrow\ldots$$
of functors $T_{n}:Ch(\mathpzc{E})\rightarrow Ch(\mathpzc{E})$ and for each $n$ a natural transformation $\psi_{n}: T_{n}\rightarrow Id_{Ch(\mathpzc{E})}$ such that 
\begin{enumerate}
\item
The diagrams 
\begin{displaymath}
\xymatrix{
T_{n}\ar[dr]^{\psi_{n}}\ar[rr] & & T_{n+1}\ar[dl]^{\psi_{n+1}}\\
& Id_{Ch(\mathpzc{E})} & 
}
\end{displaymath}
commute for all $n$
\item
for each $X\in Ch(\mathpzc{E})$ and each $n\in\mathbb{N}$, $T_{n}(X)$ is concentrated in degrees $\ge -n$.
\item
the map $T_{n}T_{n+1}\rightarrow T_{n}$ is a natural quasi-isomorphism.
\item
for each $X\in Ch(\mathpzc{E})$ and each $n\in\mathbb{N}$, $T_{n}(X)\rightarrow X$ is an admissible epimorphism in degrees $>-n$. 
\item
if $X$ is concentrated in degrees $\ge -n$, Then $T_{n}(X)\rightarrow X$ is a quasi-isomorphism.
\item
$T_{n}$ sends weak equivalences to weak equivalences.
\end{enumerate}
We write the data of a left truncation system as $(T_{n},\psi_{n})$. A left truncation system is said to be \textbf{admissible} if the maps $T_{n}(X)\rightarrow T_{n+1}(X)$ are admissible monomorphisms in degrees $>-n$, and \textbf{homologically right complete} if $\colim_{\omega}\psi_{n}$ exists and is a natural weak equivalence.
\end{defn}

\begin{defn}
Let $(T_{n},\psi_{n})$ and $(T'_{n},\psi'_{n})$ be left truncation systems. A \textbf{morphism of left truncation systems} is a collection of natural transformations $\alpha_{n}:T_{n}\rightarrow T'_{n}$ such that the diagrams below commute for all $n$
\begin{displaymath}
\xymatrix{
T_{n}\ar[d]^{\alpha_{n}}\ar[r] & T_{n+1}\ar[d]^{\alpha_{n+1}}\\
T'_{n}\ar[r] & T'_{n+1}
}
\end{displaymath}
\begin{displaymath}
\xymatrix{
T_{n}\ar[dr]^{\psi_{n}}\ar[rr]^{\alpha_{n}} & &T'_{n}\ar[dl]^{\psi'_{n}}\\
& Id_{Ch(\mathpzc{E})} & 
}
\end{displaymath}
A morphism $(\alpha_{n})$ of left truncation systems is said to be an \textbf{equivalence} if each $\alpha_{n}$ is a natural weak equivalence. 
\end{defn}

\begin{defn}
Let $\mathpzc{E}$ be an additive category which has kernels. For a complex $X_{\bullet}$ we denote by $\tau^{L}_{\ge n}X$ the complex such that $(\tau_{\ge n}X)_{m}=0$ if $m<n$, $(\tau^{L}_{\ge n}X)_{m}= X_{m}$ if $m>n$ and $(\tau^{L}_{\ge n}X)_{n}=\textrm{Ker}(d_{n})$. The differentials are the obvious ones. The construction is clearly functorial. Dually we define $\tau^{R}_{\le n}X$.
\end{defn}

\begin{example}
Let $\mathpzc{E}$ be an exact category with kernels. For each $n$ let $T_{n}=\tau^{L}_{\ge -n}$. There are compatible natural transformations $\tau^{L}_{\ge -n}\rightarrow \tau^{L}_{\ge -n+1}$, and there is a natural isomorphism $\colim_{n\in\omega}\tau^{L}_{\ge -n}\rightarrow Id_{Ch(\mathpzc{E})}$. This is called the \textbf{standard left truncation system}. Note that as explained in Proposition \ref{prop:truncquasi} above, truncation functors send quasi-isomorphisms to quasi-isomorphisms. This truncation system is in fact admissible and homologically right complete. 
\end{example}

\begin{example}
Let $\mathpzc{E}$ be an exact category, and $i_{\mathpzc{D}}:\mathpzc{D}\subset\mathpzc{E}$ an exact subcategory such that there is a $\mathpzc{D}$-left deformation functor $(Q,\eta)$ on $\mathpzc{E}$. Let $(T_{n},\psi_{n})$ be a left truncation system on $\mathpzc{E}$. Define $T_{n}^{\mathpzc{D}}\defeq Q_{+}\circ T_{n}\circ i_{\mathpzc{D}}$. Define $\overline{\psi}_{n}^{\mathpzc{D}}$ by $\psi_{n}\circ\eta_{T_{n}\circ i_{\mathpzc{D}}}$. If the inclusion $\mathpzc{D}\rightarrow\mathpzc{E}$ reflects admissible epimorphisms then this defines a left truncation system on $\mathpzc{D}$. A situation in which this occurs is that $\mathpzc{D}$ is an exact category with a projectively generating subcategory $\mathcal{P}$ which has weak kernels in the sense of \cite{bodzenta2020abelian} Section 4.5. Then by Proposition 4.20 in \cite{bodzenta2020abelian} $\mathpzc{D}$ has a right abelian envelope $\mathpzc{A}_{r}(\mathpzc{D})$. Moreover $\mathpzc{A}_{r}(\mathpzc{D})$ has enough projectives, and in fact $\mathcal{P}$ is also a projective generating subcategory of $\mathpzc{A}_{r}(\mathpzc{D})$. Thus $Ch_{+}(\mathcal{P})$ is a homological left deformation of $\mathpzc{A}_{r}(\mathpzc{D})$, and consequently  $Ch_{+}(\mathpzc{D})$ is a homological left deformation of $\mathpzc{A}_{r}(\mathpzc{D})$. The standard left truncation system on $\mathpzc{A}_{r}(\mathpzc{D})$ then induces a left truncation system on $\mathpzc{E}$. We will explain later how often this truncation system can be modified to be admissible and homologically right complete.
\end{example}

\subsection{Very Special Resolutions and Unbounded Adapted Classes}

%
%
%

\begin{defn}
\begin{enumerate}
\item
Let $\mathpzc{D}$ be an exact category.  A degree-wise split extension closed full subcategory $i_{\mathcal{B}}:\mathcal{B}\rightarrow Ch(\mathpzc{D})$ is said to be a \textbf{homological left deformation of }$Ch(\mathpzc{D})$ if there is a pair $(\overline{Q},\overline{\eta})$ where $\overline{Q}:Ch(\mathpzc{D})\rightarrow\mathcal{B}$ is a functor, and $\overline{\eta}:i_{\mathcal{B}}\circ \overline{Q}\rightarrow Id_{Ch(\mathpzc{D})}$ is a natural weak equivalence and admissible epimorphism in each degree such that whenever $X_{\bullet}$ is concentrated in degrees $\ge n$, $\overline{Q}(X_{\bullet})$ is concentrated in degrees $\ge n$.
\item
Let $\mathpzc{D}$ and $\mathpzc{E}$ be exact categories, and let $F:\mathpzc{D}\rightarrow\mathpzc{E}$  A full subcategory $i_{\mathcal{B}}:\mathcal{B}\rightarrow Ch(\mathpzc{D})$ is said to be a \textbf{homological left deformation for }$F$ if it is a homological left deformation of $Ch(\mathpzc{D})$, and $F$ preserves weak equivalences between complexes in $\mathcal{B}$.
\end{enumerate}
\end{defn}

Clearly if $i_{\mathcal{B}}:\mathcal{B}\rightarrow Ch(\mathpzc{D})$  is a left deformation of $Ch(\mathpzc{D})$, then $\mathcal{B}\cap Ch_{+}(\mathpzc{D})\rightarrow Ch_{+}(\mathpzc{D})$ is a left deformation of $Ch_{+}(\mathpzc{D})$, and consequently if $i_{\mathcal{B}}:\mathcal{B}\rightarrow Ch(\mathpzc{D})$ is a left deformation of $F:Ch(\mathpzc{D})\rightarrow Ch(\mathpzc{E})$, then $\mathcal{B}\cap Ch_{+}\rightarrow Ch_{+}(\mathpzc{D})$ is a left deformation of $F:Ch_{+}(\mathpzc{D})\rightarrow Ch_{+}(\mathpzc{E})$. In the rest of this section we will explain how often one can go in the other direction, namely we are given a left deformation of $F:Ch_{+}(\mathpzc{D})\rightarrow Ch_{+}(\mathpzc{E})$, and we wish to construct a left deformation of the functor at the level of unbounded complexes $F:Ch(\mathpzc{D})\rightarrow Ch(\mathpzc{E})$.
\begin{defn}
Let $\mathcal{B}$ be a degree-wise split extension closed full subcategory $i_{\mathcal{B}}:\mathcal{B}\rightarrow Ch(\mathpzc{D})$, and let $(T_{n},\psi_{n})$ be a left truncation system on $\mathpzc{D}$. A direct system $(P^{n}_{\bullet})_{n\in\omega}$  in $Ch(\mathpzc{E})$ is said to be a $(\mathcal{B},T,\psi)$-\textbf{very-special system} if
\begin{enumerate}
\item
Each $P^{n}_{\bullet}$ is a complex, concentrated in degrees $\ge n$.
\item
$P^{0}_{\bullet}$ is in $\mathcal{B}$.
\item
For each $n\in\omega$, the map $T_{n}P^{n}_{\bullet}\rightarrow T_{n}P^{n+1}_{\bullet}$ is a quasi-isomorphism. 
\item
For each $n\ge1$, $P_{\bullet}^{n-1}\rightarrow P_{\bullet}^{n}$ is a degree-wise split monomorphism, and its cokernel $C_{\bullet}^{n}$ is in $\mathcal{B}$. 
\end{enumerate}
An object $P_{\bullet}$ is said to be $(\mathcal{B},T,\psi)$-\textbf{very special} if it is the direct colimit of a $(\mathcal{B},T,\psi)$-very special system. The full subcategory of $(\mathcal{B},T,\psi)$-very special objects is denoted $\varinjlim^{vs}(\mathcal{B},T,\psi)$.
\end{defn}


%

The following results are generalisations of Proposition 2.63  and Corollary 2.64 in\cite{kelly2016homotopy}. The proofs are almost exactly the same as the proofs in \cite{kelly2016homotopy} (which in turn are essentially the same as the proofs of Lemma 3.3 and Theorem 3.4 in \cite{Spaltenstein}) but we include it due to the very minor alterations, and for completeness.

\begin{thm}\label{prop:DK-equivspecial}
Let $\mathpzc{E}$ be equipped with a left truncation system $(T_{n},\psi_{n})$, and let $\mathcal{B}$ be a subcategory of complexes in $Ch(\mathpzc{E})$. Suppose that there is a left deformation functor $Q_{+}:Ch_{+}(\mathpzc{E})\rightarrow\mathcal{B}$ with corresponding natural weak equivalence
$$\eta_{+}:i_{\mathcal{B}}\circ Q_{+}\rightarrow Id$$
which is a degree-wise admissible epimorphism. Then there is a $(\mathcal{B},T,\psi)$-very-special system $(P^{n}_{\bullet})_{n\ge 0}$, and a direct system of chain maps $f^{n}:P^{n}_{\bullet}\rightarrow T_{n}X_{\bullet}$ such that
\begin{enumerate}
\item
$f^{n}$ is a quasi-isomorphism for every $n\ge0$
\item
$f^{n}$ is an admissible epimorphism in each degree.
\end{enumerate}
\end{thm}


\begin{proof}
We construct the data $(P^{n}_{\bullet})_{n\ge-1}$ and $(f^{n})_{n\ge-1}$ by induction. For $n=-1$ we take $P^{-1}_{\bullet}=0$ and so $f^{-1}=0$. Let now $n\ge1$, and suppose that $P^{-1}_{\bullet},\ldots, P^{n-1}_{\bullet}$ and $f^{-1},\ldots,f^{n-1}$ have been constructed. Let $P_{\bullet}= P_{\bullet}^{n-1}$ and $Y_{\bullet}=T_{n}X_{\bullet}$. Denote by $f$ the composite $P^{n-1}_{\bullet}\rightarrow T_{n-1}X_{\bullet}\rightarrow Y_{\bullet}$. By assumption we can find a quasi-isomorphism $g:Q_{\bullet}\rightarrow\textrm{cone}(f)[1]$ which is an admissible epimorphism in each degree, and $Q_{\bullet}[-1]\in\mathcal{B}$. Now we have a degree wise splitting, $\textrm{cone}(f)[1]=P_{\bullet}\oplus Y_{\bullet}[1]$. We therefore get two maps $g':Q_{\bullet}\rightarrow P_{\bullet}$ and $g'':Q_{\bullet}\rightarrow Y_{\bullet}[1]$ which are admissible epimorphisms in each degree, and such that $g'$ is a chain map. Define $P^{n}\defeq\textrm{cone}(-g')$ and let $f^{n}:\textrm{cone}(-g')=Q[1]\oplus P\rightarrow Y$ be defined by $f^{n}=g''[1]+f$. As in \cite{Spaltenstein} Lemma 3.3, by direct calculation $f^{n}$ is a chain map and $\textrm{cone}(f^{n})=\textrm{cone}(g)[1]$. Since $g$ is a quasi-isomorphism $f^{n}$ is as well. Moreover the sequence
$$0\rightarrow P^{n-1}_{\bullet}\rightarrow P_{\bullet}^{n}\rightarrow Q_{\bullet}[1]\rightarrow0$$
is split exact in each degree. By construction the maps $T_{n}P^{n}_{\bullet}\rightarrow T_{n}P^{n+1}_{\bullet}$ are equivalences. Note that all steps in the procedure of this proof can be made functorial.
\end{proof}

\begin{cor}\label{Bres}
Let $\mathpzc{E}$ be an exact category be equipped with an admissible homologically right complete left truncation system $(T_{n},\psi_{n})$, and $\mathcal{B}$ a homological left deformation of $Ch_{+}(\mathpzc{E})$. Suppose that $\mathpzc{E}$ is weakly $(\omega,\textbf{AdMon})$-elementary. Then any chain complex $X_{\bullet}$ in $\mathpzc{E}$ admits a $\varinjlim^{vs}(\mathcal{B},T,\psi)$ resolution which is an admissible epimorphism in each degree.
\end{cor}

\begin{proof}
Fix a $(\mathcal{B},T,\psi)$-very special direct system $(P_{\bullet}^{n})_{n\ge-1}$ and a direct system of chain maps $f^{n}:P^{n}_{\bullet}\rightarrow T_{n}X_{\bullet}$ such that 
\begin{enumerate}
\item
$f^{n}$ is a quasi-isomorphism for every $n\ge0$.
\item
$f^{n}$ is an admissible epimorphism in each degree.
\end{enumerate}
Let $P_{\bullet}$ be the direct limit of the special direct system. 
Denote the kernel of the map $f^{n}$ by $K^{n}$. Then $K^{n}\rightarrow K^{n+1}$ is an admissible mono. Now we have an admissible epimorphism $f:P\rightarrow\colim_{\omega}T_{n}X$. We claim that its kernel is $K\defeq lim_{\rightarrow}K^{n}$. Fix $m\in\mathbb{Z}$. For $m>-n$ we have a diagram of exact sequences
\begin{displaymath}
\xymatrix{
0\ar[r] & K^{n}_{m}\ar[r]\ar[d] & P^{n}_{m}\ar[r]\ar[d] & (T_{n}X)_{m}\ar[r]\ar[d] & 0\\
0\ar[r] & K^{n+1}_{m}\ar[r]\ar[d] & P^{n+1}_{m}\ar[r]\ar[d] & (T_{n+1}X)_{m}\ar[r]\ar[d] & 0\\
 & \vdots & \vdots & \vdots
}
\end{displaymath}
All the vertical morphisms are admissible monomorphisms. Therefore the direct limit of this diagram is an exact sequence. But the colimit is 
$$0\rightarrow K_{m}\rightarrow P_{m}\rightarrow \colim_{\omega}(T_{n}X)_{m}\rightarrow 0$$
Thus the sequence
$$0\rightarrow K\rightarrow P\rightarrow \colim_{\omega}T_{n}X\rightarrow 0$$ 
is exact. Now each map $K_{n}\rightarrow K_{n+1}$ is an admissible monomorphism. Since $\mathpzc{E}$ is weakly $(\textbf{AdMon},\omega)$-elementary and each $K_{n}$ is exact, $K$ is also exact. Thus the composite $P\rightarrow  \colim_{\omega}T_{n}X\rightarrow X$ is an admissible epimorphism and a weak equivalence. 

\end{proof}

Let $\mathpzc{E}$ be a weakly $(\omega,\textbf{AdMon})$-elementary exact category equipped with an admissible, homologically right complete left truncation system $(T_{n},\psi_{n})$. Let $i_{\mathcal{B}}:\mathcal{B}\rightarrow Ch_{+}(\mathpzc{E})$ be a homological left deformation of $Ch_{+}(\mathpzc{E})$.  Let $\overline{T}^{\mathcal{B}}_{n}(X)$ be the functor which sends $X$ to $\overline{X}^{n}_{\bullet}$, where $(\overline{X}^{n}_{\bullet})$ is a $(\mathcal{B},T,\psi)$-very special resolution of $X$. There are by construction compatible natural transformations $\overline{\psi}^{\mathcal{B}}_{n}:\overline{X}^{n}_{\bullet}\rightarrow X$, which are admissible epimorphisms in degree $>-n$. This is an admissible homologically right complete truncation system on $\mathpzc{E}$. Note that it is equivalent to the truncation system $(T_{n},\psi_{n})$.

In particular, let $\mathpzc{E}$ be an exact category with an admissible, homologically right complete left truncation system $(T_{n},\psi_{n})$, and let $\mathpzc{D}\subset\mathpzc{E}$ be an exact full subcategory such that the inclusion reflects admissible epimorphisms and acyclicity of complexes. Suppose there is a subcategory $\mathcal{B}\subset Ch_{+}(\mathpzc{D})\subset Ch_{+}(\mathpzc{E})$ which is a homological left deformation of $Ch_{+}(\mathpzc{E})$, and consider the truncation system $(\overline{T}^{\mathcal{B}}_{n},\overline{\psi}^{\mathcal{B}}_{n})$. This restricts to an admissible left truncation system on $\mathpzc{D}$. Moreover if the inclusion $\mathpzc{D}\rightarrow\mathpzc{E}$ commutes with countable coproducts, then the restriction is a homologically right complete truncation system on $\mathpzc{D}$. 

\begin{example}
Let $\mathpzc{D}$ be an exact category with a projectively generating subcategory $\mathcal{P}$ which has weak kernels. Suppose that objects of $\mathcal{P}$ can be chosen such that for $P\in\mathcal{P}$ the functor $Hom(P,-)$ commutes with countable coproducts. Then the inclusion $\mathpzc{E}\rightarrow\mathpzc{A}_{r}(\mathpzc{E})$, where $\mathpzc{A}_{r}(\mathpzc{E})$ is the right abelian envelope of \cite{bodzenta2020abelian} commutes with countable coproducts. Thus we get an admissible, homologically right complete truncation system  $(\overline{T}^{\mathcal{P}_{n}},\overline{\psi}^{\mathcal{P}_{n}})$ on $\mathpzc{E}$.
\end{example}

%
%

%

\begin{thm}\label{lem:unboundeddeform}
Let $F:\mathpzc{D}\rightarrow\mathpzc{E}$ be a functor of exact categories, where both $\mathpzc{D}$ and $\mathpzc{E}$ are weakly $(\omega,\textbf{AdMon})$-elementary, and $\mathpzc{D}$ has a (homologically right complete, admissible) left truncation systems $(T_{n},\psi_{n})$. Suppose that $F$ commutes with countable coproducts. Let $\mathcal{B}$ be a subcategory of $Ch_{+}(\mathpzc{D})$ which is functorially left-adapted to $F$. Then $F:Ch(\mathpzc{D})\rightarrow Ch(\mathpzc{E})$ is left deformable.
\end{thm}

\begin{proof}
By Proposition \ref{prop:DK-equivspecial} it will suffice to prove that $F$ preserves weak equivalences between objects in $\colim_{\rightarrow}^{vs}(\mathcal{B},T,\psi)$. The proof is in the spirit of \cite{stacks-project} Chapter 13, Proposition 29.2. 
 Let $P_{\bullet}=\colim_{n}P_{\bullet}^{n}$ and $Q_{\bullet}=\colim_{n}Q_{\bullet}^{n}$ be $(\mathcal{B},T,\psi)$-very special objects, and let $f:P_{\bullet}\rightarrow Q_{\bullet}$ be a weak equivalence. Then $P_{\bullet}^{n}\rightarrow T_{n}P_{\bullet}$ and $Q_{\bullet}^{n}\rightarrow T_{n}Q_{\bullet}$ are equivalences. We thus get a commutative diagram in $\Ho(Ch_{+}(\mathpzc{D}))$
 \begin{displaymath}
 \xymatrix{
 P^{n}_{\bullet}\ar[d]\ar[r] & Q^{n}_{\bullet}\ar[d]\\
T_{n}P_{\bullet}\ar[r] & T_{n} Q_{\bullet}\
 }
 \end{displaymath}
 Since $\mathbb{L}F(P_{\bullet}^{n})\rightarrow F(P_{\bullet}^{n})$ and $\mathbb{L}F(Q_{\bullet}^{n})\rightarrow F(Q_{\bullet}^{n})$ are equivalences, we get an isomorphism $F(P_{\bullet}^{n})\rightarrow F(Q_{\bullet}^{n})$ in $\Ho(Ch_{+}(\mathpzc{E}))$. Moreover the diagram below is commutative in $\Ho(Ch_{+}(\mathpzc{E})$.
 \begin{displaymath}
 \xymatrix{
 F(P^{n}_{\bullet})\ar[d]\ar[r] & F(Q^{n}_{\bullet})\ar[d]\\
  F(P^{n+1}_{\bullet})\ar[r] & F(Q^{n+1}_{\bullet})
 }
 \end{displaymath}
Quasi-isomorphisms in $\mathcal{K}(\mathpzc{D})$ form a multiplicative system. Thus there is a commutative diagram in $\mathcal{K}(\mathpzc{D})$
 \begin{displaymath}
 \xymatrix{
 & R_{\bullet}^{n} \ar[dl]\ar[dr]&\\
 P^{n}_{\bullet}\ar[d] & & Q^{n}_{\bullet}\ar[d]\\
 P_{\bullet}\ar[rr] & & Q_{\bullet}
 }
 \end{displaymath}
 where both the maps are quasi-isomorphisms, and $R_{\bullet}^{n}$ may be assumed to be in $\mathcal{B}$ and be concentrated in degrees $\ge -n$. We then get a commutative diagram in $\mathcal{K}(\mathpzc{E})$
  \begin{displaymath}
 \xymatrix{
 & F(R_{\bullet}^{n})\ar[dl]\ar[dr] &\\
 F(P^{n}_{\bullet})\ar[d] & & F(Q^{n}_{\bullet})\ar[d]\\
 F(P_{\bullet})\ar[rr] && F(Q_{\bullet})
 }
 \end{displaymath}
 where the diagonal maps are quasi-isomorphisms. We therefore get commutative diagrams in $\Ho(Ch(\mathpzc{E})))$
   \begin{displaymath}
 \xymatrix{
 F(P^{n}_{\bullet})\ar[d]\ar[rr] & & F(Q^{n}_{\bullet})\ar[d]\\
 F(P_{\bullet})\ar[rr] && F(Q_{\bullet})
 }
 \end{displaymath}
 Moreover  the diagrams $F(P^{n}_{\bullet})$ and $F(Q^{n}_{\bullet})$ are in $\mathpzc{Fun}_{\textbf{AdMon}}(\omega^{op},Ch(\mathpzc{E}))$, and their colimits are $F(P_{\bullet})$ and $F(Q_{\bullet})$ respectively. Thus by the dual of Corollary \ref{cor:limithomequiv}, $F(P_{\bullet})\rightarrow F(Q_{\bullet})$ is an equivalence.

\end{proof}

\begin{rem}\label{rem:weakerres}
In the proof of the theorem, if $F$ commutes with $(\omega,\textbf{AdMon})$-colimits, then we can relax the assumptions on the very-special resolutions. Namely, we can allow resolutions of the form $\colim_{\omega}P_{\bullet}^{n}$ where each $P_{\bullet}^{n}$ is in $\mathcal{B}$, and $P_{\bullet}^{n}\rightarrow P_{\bullet}^{n+1}$ is an admissible (not necessarily degree-wise split) monomorphism such that $F(P_{\bullet}^{n})\rightarrow F(P_{\bullet}^{n+1})$ is an admissible monomorphism. 
\end{rem}

Dually one defines (admissible, homologically left complete) right truncation systems on $\mathpzc{D}$, homological right deformations of $Ch(\mathpzc{D})$ and of functors $F:Ch(\mathpzc{D})\rightarrow Ch(\mathpzc{E})$, and $\mathcal{B}$-co-very special systems. Completely formally, we get the following.

\begin{thm}\label{lem:counboundeddeform}
Let $F:\mathpzc{D}\rightarrow\mathpzc{E}$ be a functor of exact categories, where both $\mathpzc{D}$ and $\mathpzc{E}$ are weakly $(\omega,\textbf{AdEpi})$-coelementary, and $\mathpzc{D}$ has an admissible homologically left complete right truncation system. Suppose that $F$ commutes with countable products. Let $\mathcal{B}$ be a subcategory of $Ch_{-}(\mathpzc{D})$ which is functorially right-adapted to $F$. Then $F:Ch(\mathpzc{D})\rightarrow Ch(\mathpzc{E})$ is right deformable.
\end{thm}

\section{Derived Colimits (Limits) and Ind- (Pro-) Categories}

\subsection{Deriving Limits and Colimits}

Let $\mathpzc{E}$ be a complete exact category and let $\mathcal{I}$ be a filtered category. The category of functors
$$\mathpzc{Fun}(\mathcal{I}^{op},\mathpzc{E})$$
can be given the structure of an exact category, whereby a sequence of functors
$$0\rightarrow F\rightarrow G\rightarrow H\rightarrow 0$$
is exact precisely if for each $i\in\mathcal{I}$ 
$$0\rightarrow F(i)\rightarrow G(i)\rightarrow H(i)\rightarrow 0$$
is exact in $\mathpzc{E}$. 
 Consider the functor
$$\lim_{\mathcal{I}}:Ch_{-}(\mathpzc{Fun}(\mathcal{I}^{op},\mathpzc{E}))\rightarrow Ch_{-}(\mathpzc{E})$$
It turns out that if $\mathpzc{E}$ has exact products then this functor is deformable. 

Given a functor $X:\mathcal{I}^{op}\rightarrow\mathpzc{E}$, with $\mathcal{I}$ a filtered category, one can define the \textbf{Roos complex} of $X$, denoted $\mathpzc{R}(X):\mathcal{I}^{op}\rightarrow Ch_{\le0}(\mathpzc{E})$. For $i\in\mathcal{I}$,
$$\mathpzc{R}_{n}(X)(i)\defeq\prod_{i_{0}\rightarrow\ldots\rightarrow i_{n}\rightarrow i}X(i_{0})$$
The differentials are given in \cite{prosmans1999derived} Remark 3.2.6, and it is shown there that
\begin{enumerate}
\item
this construction is functorial in $X$.
\item
there is a natural weak equivalence and object-wise admissible monomorphism $X\rightarrow\mathpzc{R}(X)$
\end{enumerate}

\begin{defn}[\cite{prosmans1999derived} Definition 3.3.1]
A functor $X:\mathcal{I}^{op}\rightarrow\mathpzc{E}$ is said to be \textbf{Roos-acyclic} if the map $\lim_{\mathcal{I}^{op}}X\rightarrow\lim_{\mathcal{I}^{op}}\mathpzc{R}(X)$ is an equivalence.
\end{defn}

Prosmans proves the following for quasi-abelian categories, but the proof also works for exact categories.

\begin{prop}[\cite{prosmans1999derived} Proposition 3.3.3 ]
Let $\mathpzc{E}$ be a complete exact category with exact products. The functor $\lim_{\mathcal{I}}:\mathpzc{Fun}(\mathcal{I}^{op},\mathpzc{E})\rightarrow\mathpzc{E}$ is explicitly right-derivable by Roos-acyclic objects. 
\end{prop}

Dually, if $\mathpzc{E}$ is a cocomplete exact category with exact coproducts, then
$$\mathpzc{Fun}(\mathcal{I},\mathpzc{E})$$
can be given the structure of an exact category, and the functor
$$\colim:Ch_{+}(\mathpzc{Fun}(\mathcal{I},\mathpzc{E}))\rightarrow Ch_{+}(\mathpzc{E})$$
is explicitly left derivable. To resolve a complex, one can use the dual construction of the Roos complex, which we call the coRoos complex, and denote by $\mathpzc{cR}(F)$. 


\begin{lem}
\begin{enumerate}
\item
Let $\mathpzc{E}$ be a weakly $(\omega,\textbf{AdEpi})$-coelementary exact category with projective limits and exact products. Then for any filtered category $\mathcal{I}$,  $\mathpzc{Fun}(\mathcal{I}^{op},\mathpzc{E})$ is weakly $(\omega,\textbf{AdEpi})$-coelementary. In particular the functor
$$\lim:Ch(\mathpzc{Fun}(\mathcal{I}^{op},\mathpzc{E}))\rightarrow Ch(\mathpzc{E})$$
is right-deformable. 
\item
Let $\mathpzc{E}$ be a weakly $(\omega,\textbf{AdMon})$-elementary exact category with filtered colimits and exact coproducts. Then for any filtered category $\mathcal{I}$,  $\mathpzc{Fun}(\mathcal{I},\mathpzc{E})$ is weakly $(\omega,\textbf{AdMon})$-elementary. In particular the functor
$$\colim:Ch(\mathpzc{Fun}(\mathcal{I},\mathpzc{E}))\rightarrow Ch(\mathpzc{E})$$
is left-deformable. 
\end{enumerate}
\end{lem}

\begin{proof}
Both of these claims are easy, since exactness, admissibility of monomorphisms, and limits/ colimits are computed termwise. 
\end{proof}

\subsection{Deriving Ind- and Pro- Functors}

Let $\mathpzc{E}$ be a small category. Let $\mathpzc{Lex}(\mathpzc{E}^{op},\mathpzc{Ab})$ denote the category of left-exact functors $F:\mathpzc{E}^{op}\rightarrow\mathpzc{Ab}$, i.e. functors that send short exact sequences
$$0\rightarrow X\rightarrow Y\rightarrow Z\rightarrow 0$$
in $\mathpzc{E}$ to right exact sequences
$$0\rightarrow F(X)\rightarrow F(Y)\rightarrow F(Z)$$
is exact in $\mathpzc{Ab}$. This can be given the structure of an abelian category - see for example the remark after Lemma 2.22 in \cite{braunling2014tate}. The Yoneda embedding 
$$y:\mathpzc{E}\rightarrow \mathpzc{Lex}(\mathpzc{E}^{op},\mathpzc{Ab}),\; E\mapsto Hom(-,E)$$
defines a functor which is exact, and reflects exactness. 

\begin{defn}
The category $Ind(\mathpzc{E})$ is the full subcategory of $\mathpzc{Lex}(\mathpzc{E}^{op},\mathpzc{Ab})$  generated by the image of $y:\mathpzc{E}\rightarrow \mathpzc{Lex}(\mathpzc{E}^{op},\mathpzc{Ab})$ under filtered colimits.
\end{defn}

Any object of $Ind(\mathpzc{E})$ may be represented by a functor $\alpha:\mathcal{I}\rightarrow\mathpzc{E}$, where $\mathcal{I}$ is a filtered category. We denote such an object by $``\colim_{\rightarrow_{\mathcal{I}}}"\alpha(i)$, or just $``\colim"\alpha$ when $\mathcal{I}$ is understood. In \cite{braunling2014tate} Section 3 the Braunling and Groechenig show that $Ind^{a}(\mathpzc{E})$ (which we define later) is an extension-closed subcategory of $\mathpzc{Lex}(\mathpzc{E}^{op},\mathpzc{Ab})$, and therefore inherits an exact structure. Mutatis mutandis, their proof also works to give the following:

\begin{prop}
$Ind(\mathpzc{E})$ is an extension closed subcategory of $\mathpzc{Lex}(\mathpzc{E}^{op},\mathpzc{Ab})$. In particular $Ind(\mathpzc{E})$ inherits an exact structure from $\mathpzc{Lex}(\mathpzc{E}^{op},\mathpzc{Ab})$. Moreover, a sequence
\begin{displaymath}
\xymatrix{
0\ar[r] & X\ar[r]^{f} & Y\ar[r]^{g} & Z\ar[r]&0
}
\end{displaymath}
is exact precisely if there is a filtered category $\mathcal{I}$, functors $\alpha,\beta,\gamma:\mathcal{I}\rightarrow\mathpzc{C}$, and natural transformations $\tilde{f}:\alpha\rightarrow\beta,\tilde{g}:\beta\rightarrow\gamma$, such that
\begin{displaymath}
\xymatrix{
0\ar[r] & \alpha\ar[r]^{\tilde{f}} & \beta\ar[r]^{\tilde{g}} & \gamma\ar[r] &0
}
\end{displaymath}
such that $f\cong``\colim"\tilde{f}$, and $g\cong``\colim"\tilde{g}$. 
\end{prop}

\begin{rem}
For $\kappa$ an infinite cardinal, one can also consider the exact category $Ind_{\kappa}(\mathpzc{E})$, where we restrict to diagrams $\mathcal{I}$ which of size at most $\kappa$. What is proven in  \cite{braunling2014tate} is in fact the slightly stronger result that for each $\kappa$, $Ind^{a}_{\kappa}(\mathpzc{E})$ is an extension closed subcategory of $\mathpzc{Lex}(\mathpzc{E}^{op},\mathpzc{Ab})$. Again this works for $Ind_{\kappa}(\mathpzc{E})$.
\end{rem}

The exact structure on $Ind(\mathpzc{E})$ is called the \textbf{indization} of the exact structure on $\mathpzc{E}$.  From now on we shall assume that $Ind(\mathpzc{E})$ has a left truncation system, which we may assume to be admissible and right homologically complete. This is the case, for example, if $\mathpzc{E}$ (and therefore $Ind(\mathpzc{E})$) has kernels. It is also the case if $\mathpzc{E}$ has a projective generating set $\mathcal{P}$ with weak kernels, because $\mathcal{P}$ is still a projective generating set in $\mathpzc{E}$. 

\begin{prop}
For any filtered category $\mathcal{I}$ the functor $\mathpzc{Fun}(\mathcal{I},Ind(\mathpzc{E}))\rightarrow Ind(\mathpzc{E})$ exists and is exact. In particular it is a strongly co-ML category. If all cokernels exist then it is also cocomplete.
\end{prop}

\begin{proof}
Filtered colimits are exact in $\mathpzc{Lex}(\mathpzc{E}^{op},\mathpzc{Ab})$, and $Ind(\mathpzc{E})$ is an extension closed subcategory which is closed under filtered colimits in $\mathpzc{Lex}(\mathpzc{E}^{op},\mathpzc{Ab})$. So in fact all filtered colimits are exact in $Ind(\mathpzc{E})$.
\end{proof}

Thus for any filtered category $\mathcal{I}$ the colimit functor
$$\colim:Ch(\mathpzc{Fun}(\mathcal{I},Ind(\mathpzc{E})))\rightarrow Ch(Ind(\mathpzc{E}))$$
sends equivalences to equivalences, and is trivially deformable.

Dually one defines the category $Pro(\mathpzc{E})\defeq (Ind(\mathpzc{E}^{op}))^{op}$. This is a strongly ML-category, and if all kernels exist then it is complete. For any filtered category $\mathcal{I}$
$$\lim:Ch(\mathpzc{Fun}(\mathcal{I}^{op},Pro(\mathpzc{E})))\rightarrow Ch(Pro(\mathpzc{E}))$$
sends equivalences to equivalences, and is trivially deformable.
\begin{notation}
Let $F:\mathpzc{E}\rightarrow\mathpzc{D}$ be a functor of exact categories.
\begin{enumerate}
\item
We denote by $IF$ the functor $Ind(\mathpzc{E})\rightarrow Ind(\mathpzc{D})$ sending $``\colim_{\mathcal{I}}"X_{i}$  to $``\colim_{\mathcal{I}}"F(X)_{i}$ 
\item
We denote by $PF$ the functor $Pro(\mathpzc{E})\rightarrow Pro(\mathpzc{D})$ sending $``\lim_{\mathcal{I}^{op}}"X_{i}$  to $``\lim_{\mathcal{I}^{op}}"F(X)_{i}$ 
\item
If $\mathpzc{D}$ has filtered colimits we denote by $\overline{F}:Ind(\mathpzc{E})\rightarrow\mathpzc{D}$ the functor which sends an object $``\colim_{\mathcal{I}}"X_{i}$ to $\colim_{\mathcal{I}}F(X_{i})$.
\item
If $\mathpzc{D}$ has projective colimits we denote by $\overline{F}:Pro(\mathpzc{E})\rightarrow\mathpzc{D}$ the functor which sends an object $``\lim_{\mathcal{I}^{op}}"X_{i}$ to $\lim_{\mathcal{I}^{op}}F(X_{i})$.
\end{enumerate}
\end{notation}

\begin{rem}
Let $F:\mathpzc{E}\rightarrow\mathpzc{D}$ be a functor of exact categories with $\mathpzc{D}$ cocomplete. Then there is a natural equivalence $\overline{F}\cong\overline{Id}_{\mathpzc{D}}\circ IF$. 
\end{rem}

In this section we will be interested in deriving functors of the form $\overline{F}$. We need to introduce some classes of objects.

%

\begin{notation}
Let $\mathpzc{D}$ be an exact category, and $i_{\mathcal{S}}:\mathcal{S}\rightarrow\mathpzc{D}$ a full subcategory. We denote by
\begin{enumerate}
\item
$Ind^{a}(\mathcal{S})$ is the full subcategory of $Ind(\mathpzc{E})$ consisting of those objects which are isomorphic to objects of the form $``\colim_{\mathcal{I}}"X_{i}$ where for $i\rightarrow j$ in $\mathcal{I}$, $X_{i}\rightarrow X_{j}$ is an admissible monomorphism, and each $X_{i}$ is in $\mathcal{S}$.
\item
$\coprod(\mathcal{S})$ the full subcategory of $Ind^{a}(\mathpzc{E})$ consisting of objects which are isomorphic to coproducts of objects of $\mathcal{S}$.
\item
If $F:\mathpzc{D}\rightarrow\mathpzc{E}$ is a functor, by $Ind^{F-wcfl}(\mathcal{S})$ is the full subcategory of $Ind^{a}(\mathcal{S})$ consisting of those objects which are isomorphic to objects of the form $``\colim_{\mathcal{I}}"X_{i}$ where the diagram $F\circ X:\mathcal{I}\rightarrow\mathpzc{E}$ is weakly co-flasque
\item
$Pro^{a}(\mathcal{S})$ the full subcategory of $Pro(\mathpzc{E})$ consisting those objects which are isomorphic to objects of the form $``\lim_{\mathcal{I}^{op}}"X_{i}$ where for $i\rightarrow j$ in $\mathcal{I}$, $X_{j}\rightarrow X_{i}$ is an admissible epimorphism, and each 
\item
$\prod(\mathcal{S})$ the full subcategory of $Pro^{a}(\mathpzc{E})$ consisting of objects which are isomorphic to products of objects of $\mathcal{S}$.
\item
If $\mathpzc{E}$ is complete $Pro^{F-wfl}(\mathcal{S})$ is the full subcategory of $Pro^{a}(\mathcal{S})$ consisting of those objects which are isomorphic to objects of the form $``\lim_{\mathcal{I}^{op}}"X_{i}$ where the diagram $F\circ X:\mathcal{I}^{op}\rightarrow\mathpzc{E}$ is weakly flasque
\end{enumerate}
\end{notation}

\begin{rem}
$\coprod(\mathcal{S})\subset Ind^{F-wcfl}(\mathcal{S})$ and $\prod(\mathcal{S})\subset Pro^{F-wfl}(\mathcal{S})$ for any functor $F:\mathpzc{D}\rightarrow\mathpzc{E}$.
\end{rem}

In particular, $Ind^{a}(\mathpzc{E})$ is the full subcategory of $Ind(\mathpzc{E})$ consisting of \textbf{essentially admissibly monomorphic objects}. By \cite{braunling2014tate} Section 3 this is in fact itself an exact subcategory of $Ind(\mathpzc{E})$.

\begin{prop}\label{prop:ind-resolution}
Let $X=``\colim_{\mathcal{I}}"X_{i}\in Ind(\mathpzc{E})$. Then one can functorially construct an exact sequence.
$$0\rightarrow K\rightarrow\overline{X}\rightarrow X\rightarrow 0$$
where $\overline{X}$ is in $\coprod(\mathpzc{E})$ and $K$ is in $Ind^{a}(\mathpzc{E})$. In particular $Ch_{+}(\coprod(\mathpzc{E}))$ and $Ch_{+}(Ind^{a}(\mathpzc{E}))$ are left deformations of $Ch_{+}(Ind(\mathpzc{E}))$.
\end{prop}

\begin{proof}
Write $\overline{X}=\bigoplus_{i\in\mathcal{I}}X_{i}$. We may assume that $\mathcal{I}$ is a directed set. Then the canonical map $\pi:\overline{X}\rightarrow ``\colim_{\mathcal{I}}"X_{i}$ is an admissible epimorphism. Let $\mathcal{I}^{<\infty}$ denote the set of finite subsets of $\mathcal{I}$ which have a maximal element. As an ind-object $\overline{X}$ can be written as $``\colim_{I\in\mathcal{I}^{<\infty}}"\bigoplus_{i\in I}X_{i}$. For $I\subset J\subset\mathcal{I}$ with $J\in\mathcal{I}^{<\infty}$, the map $\bigoplus_{i\in I}X_{i}\rightarrow \bigoplus_{j\in J}X_{j}$ is a split, and hence admissible monomorphism. Moreover the map $\overline{X}\rightarrow X$ is induced by the canonical sum maps $\bigoplus_{i\in I}X_{i}\rightarrow X_{max\{i\in I\}}$ for $I\in\mathcal{I}^{<\infty}$. This is a split map, which shows that $\overline{X}\rightarrow X$ is an admissible epimorphism. Now there is a diagram $\mathcal{F}$ together with final functors $p:\mathcal{F}\rightarrow\{\mathcal{I}^{<\infty}\}$ n and $q:\mathcal{F}\rightarrow\mathcal{I}$, and compatible admissible epimorphisms $\pi_{f}:\overline{X}_{p(f)}\rightarrow X_{q(f)}$ such that $\pi=``\colim_{\mathcal{F}}"\pi_{f}$. Then $K\cong``\colim_{\mathcal{F}}"Ker(\pi_{f})$. Moreover for $f\rightarrow f'\in\mathcal{F}$ we have a commutative diagram of exact sequences
\begin{displaymath}
\xymatrix{
0\ar[r] & K_{f}\ar[d]\ar[r] & \overline{X}_{p(f)}\ar[d]\ar[r] & X_{q(f)}\ar[d]\ar[r] & 0\\
0\ar[r] & K_{f'}\ar[r] & \overline{X}_{p(f')}\ar[r] & X_{q(f')}\ar[r] & 0
}
\end{displaymath}
Since the middle vertical map is an admissible monomorphism, the obscure axiom implies that the left vertical map is as well. Thus $K$ is in $Ind^{a}(\mathpzc{E})$.
\end{proof}

\begin{prop}\label{prop:co-Mlacyc}
Let $\mathpzc{E}$ be a co-ML category, let $\mathcal{I}$ be a filtered category, and let $X$ be an object in $\mathpzc{Fun}_{\textbf{AdMon}}^{wcfl}(\mathcal{I},\mathpzc{E})$. Then the map $\mathbb{L}\colim_{\mathcal{I}}X\rightarrow\colim_{\mathcal{I}}X$ is an equivalence.
\end{prop}

\begin{proof}
Consider the resolution $\mathpzc{cR}(X)\rightarrow X$. This augmented complex is a complex of objects in $\coprod(\mathpzc{E})$. In particular it is in $\mathpzc{Fun}_{\textbf{AdMon}}^{wcfl}(\mathcal{I},\mathpzc{E})$. Thus the augmented complex $\colim\mathpzc{cR}(X)\rightarrow X$ is acyclic, i.e. $\mathbb{L}\colim_{\mathcal{I}}X\rightarrow\colim_{\mathcal{I}}X$ is an equivalence.
\end{proof}

\begin{cor}
Let $\mathpzc{E}$ be co-ML, let $\mathcal{I}$ be a filtered category, and let
$$0\rightarrow X\rightarrow Y\rightarrow Z\rightarrow 0$$
be an exact sequence in $\mathpzc{Fun}(\mathcal{I},\mathpzc{E})$ with $Z$ in $\mathpzc{Fun}_{\textbf{AdMon}}^{wcfl}(\mathcal{I},\mathpzc{E})$. Then the sequence
$$0\rightarrow\colim_{\mathcal{I}}X\rightarrow\colim_{\mathcal{I}}Y\rightarrow\colim_{\mathcal{I}}Z\rightarrow 0$$
is exact in $\mathpzc{E}$. Moreover if $Y\in\mathpzc{Fun}_{\textbf{AdMon}}^{wcfl}(\mathcal{I},\mathpzc{E})$ then so is $X$.
\end{cor}

\begin{proof}
The first claim is immediate from Proposition \ref{prop:co-Mlacyc}. For the second, let $J\subset\mathcal{I}$ be directed. Since the restriction of a weakly co-flasque system to a directed subset is still weakly co-flasque, we get a diagram of exact sequences
\begin{displaymath}
\xymatrix{
0\ar[r] & \colim_{J}X|_{J}\ar[d]\ar[r] &  \colim_{J}Y|_{J}\ar[d]\ar[r] &  \colim_{J}Z|_{J}\ar[d]\ar[r] & 0\\
0\ar[r] & \colim_{\mathcal{I}}X\ar[r] & \colim_{\mathcal{I}}Y\ar[r] & \colim_{\mathcal{I}}Z\ar[r] & 0
}
\end{displaymath}
The second two vertical maps are admissible monomorphisms, so the first must be as well.
\end{proof}

\begin{cor}
Let $F:\mathpzc{E}\rightarrow\mathpzc{D}$ be an exact functor of exact categories where $\mathpzc{D}$ is co-ML, and let
$$0\rightarrow ``\colim_{\mathcal{I}}"X_{i}\rightarrow ``\colim_{\mathcal{J}}"Y_{j}\rightarrow ``\colim_{\mathcal{K}}"Z_{k}\rightarrow 0$$
be an exact sequence in $Ind(\mathpzc{E})$ with $Z$ in $Ind^{F-wcfl}(\mathpzc{E})$. Then the sequence
$$0\rightarrow \colim_{\mathcal{I}}X_{i}\rightarrow \colim_{\mathcal{J}}Y_{j}\rightarrow \colim_{\mathcal{K}}Z_{k}\rightarrow 0$$
is exact in $\mathpzc{D}$. Moreover if $`\colim_{\mathcal{J}}"Y_{j}$ and $``\colim_{\mathcal{K}}"Z_{k}$ are in $Ind^{F-wcfl}(\mathpzc{D})$, then so is  $``\colim_{\mathcal{I}}"X_{i}$.
\end{cor}

\begin{proof}
By picking  cofinal functors $\mathcal{L}\rightarrow\mathcal{I},\mathcal{L}\rightarrow\mathcal{J},\mathcal{L}\rightarrow\mathcal{K}$, we may assume that the sequence arises from an exact sequence 
$$0\rightarrow X\rightarrow Y\rightarrow Z\rightarrow 0$$
in $\mathpzc{Fun}(\mathcal{L},\mathpzc{E})$. 
\end{proof}

We will denote by $(Q^{wcfl},\eta^{wcfl})$ the $Ind^{wcfl}(\mathpzc{E})$-left deformation functor constructed in the proof of Proposition \ref{prop:ind-resolution}. Namely $Q^{wcfl}(X)=\overline{X}$, and $\eta_{X}^{wcfl}:\overline{X}\rightarrow X$ is the natural epimorphism.

\begin{prop}\label{prop:indsdef}
Let $\mathpzc{E}$ be an exact category, $i_{\mathcal{S}}:\mathcal{S}\rightarrow\mathpzc{E}$ a full subcategory, and $(Q,\eta)$ an $\mathcal{S}$-left deformation functor. Then there exists a $\coprod(\mathcal{S})$-left deformation functor on $Ind(\mathpzc{E})$.
\end{prop}

\begin{proof}
Write $Q^{wcfl,\mathcal{S}}\defeq Q^{wcfl}\circ IQ$, where $IQ$ is the functor $Ind(\mathcal{E})\rightarrow Ind(\mathcal{S})$ extending $Q$ by formal filtered colimits. Since $\mathcal{S}$ is closed under finite sums, this defines a functor $Ind(\mathpzc{E})\rightarrow \coprod(\mathcal{S})$. 
\end{proof}

\begin{prop}\label{prop:indprobounded}
 Let $\mathpzc{D}$ be a co-ML exact category, and let $F:\mathpzc{E}\rightarrow\mathpzc{D}$ be a functor. Suppose there is a category $\mathcal{S}\subset\mathpzc{E}$ which is left functorially adapted to $F$. Then the functor
 $$\overline{F}:Ch(Ind(\mathpzc{E}))\rightarrow Ch(\mathpzc{D})$$
 is left deformable. If $F$ is exact, i.e. $\mathcal{S}=\mathpzc{E}$, and $\mathpzc{D}$ is strongly co-ML, then for $X\in Ind(\mathpzc{E})$, $\mathbb{L}\overline{F}(X)$ is quasi-isomorphic to a complex concentrated in degrees $0,1$.
\end{prop}

\begin{proof}
Consider the category $Ind^{F-wcfl}(\mathcal{S})$. Clearly if
$$0\rightarrow X\rightarrow Y\rightarrow Z\rightarrow 0$$
is an exact sequence with $Y,Z\in Ind^{F-wcfl}(\mathcal{S})$, then $X\in Ind^{F-wcfl}(\mathcal{S})$. Moreover since $\overline{F}\cong\overline{Id}_{\mathpzc{D}}\circ IF$, $\overline{F}$ sends exact sequences in $Ind^{F-wcfl}(\mathcal{S})$ to exact sequences in $\mathpzc{D}$. Finally since $\coprod\mathcal{S}\subset Ind^{F-wcfl}(\mathcal{S})$, $Ch_{+}(Ind^{F-wcfl}(\mathcal{S}))$ is a left deformation of $Ch_{+}(Ind(\mathpzc{E}))$, and by what we have just shown a left deformation of $\overline{F}$. By Lemma \ref{lem:unboundeddeform}, $\colim^{vs}_{\rightarrow}(Ch_{+}(Ind^{wcfl}(\mathcal{S}),T,\psi)$ is a left deformation of $\overline{F}$

For the final claim, we note that in this case any object $X$ has a resolution 
$$0\rightarrow K\rightarrow\overline{X}\rightarrow X\rightarrow 0$$
where $K$ and $\overline{X}$ are in $Ind^{a}(\mathpzc{E})$. Thus $\mathbb{L}\overline{F}(X)$ is quasi-isomorphic to the two-term complex $\overline{F}(K)\rightarrow\overline{F}(\overline{X})$ with $\overline{F}(K)$ in degree $1$. 
\end{proof}

Entirely dually, we have the following

\begin{prop}\label{prop:probounded}
 Let $\mathpzc{D}$ be a quasi-abelian category, and let $F:\mathpzc{E}\rightarrow\mathpzc{D}$ be a functor where $\mathpzc{D}$ is ML. Suppose there is a category $\mathcal{S}\subset\mathpzc{E}$ which is right functorially adapted to $F$. Then the functor
 $$\overline{F}:Ch(Pro(\mathpzc{E}))\rightarrow Ch(\mathpzc{D})$$
 is right derivable.  If $F$ is exact, i.e. $\mathcal{S}=\mathpzc{E}$, and $\mathpzc{D}$ is strongly ML, then for any $X\in Pro(\mathpzc{E})$, $\mathbb{R}\overline{F}(X)$ is quasi-isomorphic to a complex concentrated in degrees $0,-1$.
\end{prop}

As shown in \cite{prosmans1999derived} Corollary 7.3.7 for quasi-abelian categories, deriving functors out of $Ind(\mathpzc{E})$ (resp. $Pro(\mathpzc{E}))$ is closely connected to computing derived colimits (resp. derived limits). Indeed, let $X$ be an object of $Ind(\mathpzc{E})$. Let $\mathcal{I}$ be a filtered category, and $G_{X}:\mathcal{I}\rightarrow Ind(\mathpzc{E})$ a functor which factors through $\mathpzc{E}$, such that $X\cong\colim G_{X}$. Since filtered colimits are exact in $Ind(\mathpzc{E})$, there is an equivalence
$$\mathbb{L}\colim G_{X}\cong X$$
Thus, if $\mathpzc{cR}(G_{X})$ is the coRoos complex of $G_{X}$ in $Ind(\mathpzc{E})$, then the admissible epimorphism $\colim\mathpzc{cR}(G_{X})\rightarrow X$ is an equivalence. Let $F:\mathpzc{E}\rightarrow\mathpzc{D}$ be an exact functor. Now each of the terms in the complex $\colim\mathpzc{cR}(G_{X})$ is a coproduct of objects in $\mathpzc{E}$, and is therefore an object of $Ind^{F-wcfl}(\mathpzc{E})$. Hence each term is $\overline{F}$-acyclic. So we have
$$\mathbb{L}\overline{F}(X)\cong\overline{F}(\colim\mathpzc{cR}(G_{X}))\cong\colim\mathpzc{cR}(F\circ G_{X})\cong\mathbb{L}\colim F\circ G_{X}$$
We have proven the following.

\begin{prop}
 Let $\mathpzc{D}$ be an exact category, and let $F:\mathpzc{E}\rightarrow\mathpzc{D}$ be an exact functor. Suppose that $\mathpzc{D}$ is co-ML. Let $X=``\colim_{\mathcal{I}}"X_{i}$ be an object of $Ind(\mathpzc{E})$. Then there is a natural equivalence
 $$\mathbb{L}\overline{F}(X)\cong\mathbb{L}\colim_{\mathcal{I}}F(X_{i})$$
\end{prop}

Dually we have:

\begin{prop}
 Let $\mathpzc{D}$ be an exact category, and let $F:\mathpzc{E}\rightarrow\mathpzc{D}$ be a functor. Suppose that $\mathpzc{D}$ is ML. Let $X=``\lim_{\mathcal{I}}"X_{i}$ be an object of $Pro(\mathpzc{E})$. Then there is a natural equivalence
 $$\mathbb{R}\overline{F}(X)\cong\mathbb{R}\lim_{\mathcal{I}^{op}}F(X_{i})$$
\end{prop}

\subsubsection{Essentially Monomorphic Objects}

In this section we shall assume that all categories have kernels, and are equipped with the standard left truncation system. 

\begin{defn}
$Ind^{m}(\mathpzc{E})$ is the full subcategory of $Ind(\mathpzc{E})$ consisting of those objects which are isomorphic to objects of the form $``\colim_{\mathcal{I}}"X_{i}$ where for $i\rightarrow j$ in $\mathcal{I}$, $X_{i}\rightarrow X_{j}$ is a monomorphism. 
\end{defn}

Again the proof in \cite{braunling2014tate} Section 3 that $Ind^{a}(\mathpzc{E})\subset \mathpzc{Lex}(\mathpzc{E}^{op},\mathpzc{Ab})$ is extension closed can be straightforwardly modified to show the following. 

\begin{prop}
$Ind^{m}(\mathpzc{E})$ is an extension closed subcategory of $Ind(\mathpzc{E})$. In particular it inherits an exact category structure from $Ind(\mathpzc{E})$. 
\end{prop}

Since $Ind^{a}(\mathpzc{E})$ is a full subcategory of $Ind^{m}(\mathpzc{E})$ a proof identical to the one for Proposition \ref{prop:indprobounded} gives the following.

\begin{prop}\label{prop:mindprobounded}
 Let $\mathpzc{D}$ be co-ML, and let $F:\mathpzc{E}\rightarrow\mathpzc{D}$ be an exact functor. Then the functor
 $$\overline{F}:Ch(Ind^{m}(\mathpzc{E}))\rightarrow Ch(\mathpzc{D})$$
 is left deformable. Moreover if $\mathpzc{D}$ is strongly co-ML then for any $X\in Ind^{m}(\mathpzc{E})$, $\mathbb{L}\overline{F}(X)$ is quasi-isomorphic to a complex concentrated in degrees $0,1$.
\end{prop}

\begin{prop}\label{prop:indcolimiderived}
Let $\mathpzc{D}$ be co-ML, and let $F:\mathpzc{E}\rightarrow\mathpzc{D}$ be an exact functor. Let $X=``\colim_{\mathcal{I}}"X_{i}$ be an object of $Ind^{m}(\mathpzc{E})$. Then there is a natural equivalence
 $$\mathbb{L}\overline{F}(X)\cong\mathbb{L}\colim_{\mathcal{I}}F(X_{i})$$
\end{prop}

\subsubsection{Right Deriving Functors of Ind- Categories}

In the previous sections we showed that if $\mathpzc{E}$ is an exact category, then $Ind(\mathpzc{E})$ and $Ind^{m}(\mathpzc{E})$ are, essentially completely formally, weakly co-ML categories, with the former being strongly co-ML. Consequently it is often relatively easy to construct unbounded left derived functors into and out of such categories. In this section we deal briefly with the far less formal matter of right-deriving functors of Ind- categories. The major issues are that projective limits in these categories may not exist, and if they do there is no guarantee that they will satisfy any exactness properties. One way to solve the latter problem is with enough projectives.

\begin{prop}
If an exact category $\mathpzc{E}$ has enough projectives then so do $Ind(\mathpzc{E})$ and $Ind^{m}(\mathpzc{E})$. In particular they are weakly ML
\end{prop}

Whether the categories $Ind(\mathpzc{E})$ or $Ind^{m}(\mathpzc{E})$ are complete is a trickier question. However the situation is simpler when $\mathpzc{E}$ is quasi-abelian. Indeed let $\mathpzc{E}$ be a quasi-abelian category with enough projectives. Then $Ind(\mathpzc{E})$ is an elementary quasi-abelian category in the terminology of \cite{qacs} Definition 2.1.10. In particular it is complete by \cite{qacs} Proposition 2.1.15. Suppose now we equip $\mathpzc{E}$ with an exact structure $\mathcal{Q}$ which may not coincide with the quasi-abelian one, but which still has enough projectives. Since $Ind(\mathpzc{E})$ is complete and ML, for any filtered category $\mathcal{I}$ we get unbounded derived limit functors
$$\mathbb{R}\lim:Ch(\mathpzc{Fun}(\mathcal{I}^{op},Ind(\mathpzc{E})))\rightarrow Ch(Ind(\mathpzc{E}))$$
Furthermore if $Ind^{m}(\mathpzc{E})$ is cocomplete, then it is a reflective subcategory of $Ind(\mathpzc{E})$. Indeed a left adjoint to the inclusion $Ind^{m}(\mathpzc{E})\rightarrow Ind(\mathpzc{E})$ is given by sending $``\colim_{\mathcal{I}}"X_{i}$ to the colimit  $\colim_{\mathcal{I}}X_{i}$ computed in $Ind^{m}(\mathpzc{E})$. Therefore it  is complete, and therefore it is both ML and strongly co-ML. Once again we get an unbounded derived limit functor.
$$\mathbb{R}\lim:Ch(\mathpzc{Fun}(\mathcal{I}^{op},Ind^{m}(\mathpzc{E})))\rightarrow Ch(Ind^{m}(\mathpzc{E}))$$

\begin{example}\label{ex:locallysplit}
Let $\mathpzc{E}$ be a quasi-abelian category with enough projectives. Following \cite{mukherjee2020topological} Chapter 2, Section 3, the \textbf{locally split exact structure} on $Ind(\mathpzc{E})$/ $Ind^{m}(\mathpzc{E})$ is the indization of the split exact structure on $\mathpzc{E}$. Moreover in the split exact structure on $\mathpzc{E}$ all objects are projective, so both $Ind(\mathpzc{E})$ and $Ind^{m}(\mathpzc{E})$ have enough projectives. Thus if $Ind(\mathpzc{E})$ (resp. $Ind^{m}(\mathpzc{E})$)is complete then it is ML.
\end{example}

\section{Example: Bornological and Locally Convex Modules}

In this section we apply our results to examples from functional analysis, specifically categories of bornological spaces and locally convex topological spaces. Homological algebra for functional analysis has been developed in detail in e.g. \cite{reconstruction}, \cite{dcfapp}, \cite{qacs}, \cite{prosmans2000homological}, \cite{wengenroth2003derived}.

\begin{defn}
A \textbf{semi-normed ring} is a unital commutative ring $R$, equipped with a function $|-|:R\rightarrow\mathbb{R}_{\ge0}$ such that for $s,r\in R$
\begin{enumerate}
\item
$|s+r|\le |s|+|r|$
\item
$|sr|\le |s||r|$
\item
$|0|=0$. 
\end{enumerate}
A semi-normed ring is said to be a \textbf{normed ring} if $|r|=0\Rightarrow r=0$. A normed ring is said to be a \textbf{Banach ring} if it is complete as a metric space, with the metric induced by $|-|$. 
\end{defn}

For example any discrete valuation ring $R$ is naturally a semi-normed ring. 

\begin{defn}
Let $R$ be a semi-normed ring. A \textbf{semi-normed } $R$-\textbf{module} is a $R$-module $M$, equipped with a function $||-||:M\rightarrow\mathbb{R}_{\ge0}$ such that there is $C_{M}\ge0$, such that for all $r\in R,m,n\in M$
\begin{enumerate}
\item
$||rm||\le C_{M}|r|||m||$
\item
$||m+n||\le  ||m||+||n||$
\item
$||0||=0$
\end{enumerate}
A map $f:M\rightarrow N$ of semi-normed $R$-modules is said to be \textbf{bounded} if there is $C_{f}\ge 0$ such that for all $m\in M$,
$$||f(m)||\le C_{f}||m||$$
\end{defn}
Semi-normed $R$-modules arrange into a category $\sNorm_{R}$. 

\begin{defn}
An object $V\in\sNorm_{R}$ is said to be \textbf{torsion-free} if it is torsion-free as an $R$-module. The full subcategory of $\sNorm_{R}$ consisting of torsion-free modules is denoted $\sNorm_{R}^{tf}$. 
\end{defn}

We are gong to show that $\sNorm_{R}$ and $\sNorm_{R}^{tf}$ are quasi-abelian categories with enough projectives. For a collection $\{M_{i}\}_{i\in\mathcal{I}}$ of semi-normed $R$-modules, we denote by $\coprod^{\le1}_{i\in\mathcal{I}}M_{i}$ the semi-normed $R$-module whose underlying $R$-module is $\coprod_{i\in\mathcal{I}}M_{i}$, with semi-norm $\rho((m_{i}))=\sum_{i\in\mathcal{I}}\rho_{i}(m_{i})$, where $\rho_{i}$ is the semi-norm on $M_{i}$. For a semi-normed $R$-module $M$,  and $\epsilon>0$, we also denote by $M_{\epsilon}$ the semi-normed $R$-module with $\rho_{M_{\epsilon}}\defeq\epsilon\rho_{M}$. Finally, we denote by $M^{\rightarrow 0}$ the module $\coprod^{\le1}_{\epsilon\in\mathbb{R}_{>0}}M_{\epsilon}$. The details of the following will be given in forthcoming work. The proof of existence of projectives is very similar to \cite{qacs} Proposition 3.2.11.

\begin{prop}
Both $\sNorm_{R}$ and $\sNorm_{R}^{tf}$ are quasi-abelian categories with enough functorial projectives. Moreover a sequence is exact in $\sNorm_{R}^{tf}$ if and only if it is exact in $\sNorm_{R}$.
\end{prop}

\begin{proof}[Sketch proof]
The proof that they are quasi-abelian is essentially Lemma 2.3.2 in \cite{mukherjee2020topological}. One simply has to observe that in all the constructions there everything remains torsion-free. Let us prove first that $\sNorm_{R}$ has enough functorial projectives. Let $M$ be a semi-normed $R$-module. Consider the $R$-module $P(M)$ defined by
$$P(M)\defeq\Bigr(\coprod^{\le1}_{m:\rho_{M}(m)\neq0}R_{\rho_{M}(m)}\Bigr)\oplus\coprod^{\le1}_{m:\rho_{M}(m)=0}R^{\rightarrow 0}$$
Then the obvious projection map $P(M)\rightarrow M$ is an admissible epimorphism. Now the finclusion functor $\sNorm_{R}^{tf}\rightarrow\sNorm_{R}$ is exact. Moreover it has a left adjoint $(-)^{tf}$ given by quotienting out by the torsion submodule. Thus for any projective $R$-module $P$, $P^{tf}$ is a projective torsion-free $R$-module. So if $M\in\sNorm_{R}^{tf}$, then we get an admissible epimorphism $P(M)^{tf}\rightarrow M$ where $P(M)^{tf}$ is projective. 
\end{proof}

\subsection{Bornological $R$-Modules}

Let $R$ be a semi-normed ring.

\begin{defn}
\begin{enumerate}
\item
The category of \textbf{bornological} $R$-\textbf{modules} is defined to be $\mathpzc{Born}_{R}\defeq Ind^{m}(\sNorm_{R})$. 
\item
The category of \textbf{torsion-free bornological} $k$-\textbf{modules} is defined to be $\mathpzc{Born}^{tf}_{R}\defeq Ind^{m}(\sNorm^{tf}_{R})$. 
\end{enumerate}
\end{defn}

\begin{prop}
The categories $\mathpzc{Born}_{R}$ and $\mathpzc{Born}^{tf}_{R}$ are complete and cocomplete.
\end{prop}

\begin{proof}
The functor $(-)^{tf}:\sNorm_{R}\rightarrow\sNorm_{R}^{tf}$ extends to a functor $(-)^{tf}:\mathpzc{Born}_{R}\rightarrow\mathpzc{Born}^{tf}_{R}$ which is left adjoint to the inclusion functor, i.e. $\mathpzc{Born}^{tf}_{R}$ is a reflective subcategory of $\mathpzc{Born}_{R}$. It thus suffices to show that $\mathpzc{Born}_{R}$ is complete and cocomplete. This is in \cite{bambozzi2016dagger} Section 1.4.
\end{proof}

\begin{cor}
$\mathpzc{Born}_{R}$ and $\mathpzc{Born}^{tf}_{R}$ are complete, and cocomplete categories. Moreover for both the quasi-abelian and locally split exact structures, they are ML and strongly co-ML
\end{cor}

As a consequence of this corollary and Proposition \ref{prop:probounded} we get the following, which is Proposition 3.4.4 in \cite{mukherjee2020topological}

\begin{cor}
The complete and cocomplete category $\mathpzc{Born}^{tf}_{R}$ is ML and co-ML for both the quasi-abelian and locally split exact structures. Thus for both of these exact structures the functor
$$Ch(Pro(\mathpzc{Born}^{tf}_{R}))\rightarrow Ch(\mathpzc{Born}^{tf}_{R})$$
is right-deformable. 
\end{cor}

\subsection{Locally Convex $k$-Modules}

Let $R$ be Banach ring, and consider the category $\mathcal{T}_{c,R}$ of locally convex topological $R$-modules. An object $X$ of $\mathcal{T}_{c,R}$ is a topological $R$-module whose topology is determined by a family of semi-norms $P_{X}$. The same proof as that of Proposition 1.1.10 and Proposition 1.2.2 in \cite{dcfapp} gives the following.

\begin{prop}
$\mathcal{T}_{c,R}$ is a complete and cocomplete quasi-abelian category with exact products and exact coproducts. 
\end{prop}

In \cite{dcfapp} Schneiders gives a useful condition implying the acyclicity of a projective system in $\mathcal{T}_{c,R}$.

\begin{defn}[Definition 1.2.6 \cite{dcfapp}]
Let $\mathcal{I}$ be a filtered category. A functor $X:\mathcal{I}^{op}\rightarrow\mathcal{T}_{c,R}$ is said to \textbf{satisfy condition SC} if for any $i\in\mathcal{I}$ and any absolutely convex neighbourhood $U$ of zero in $X_{i}$, there is $j\ge i$ such that 
$$x_{i,k}(X_{k})\subset q_{i}(\lim_{\mathcal{I}}X_{s})+U,\;\;\forall k\ge j$$
where $x_{i,k}:X_{k}\rightarrow X_{i}$ is the transition map and $q_{i}:\lim_{s\in\mathcal{I}}X_{s}\rightarrow X_{i}$ is the canonical projection.
\end{defn}

We denote by $|-|_{t}:\mathcal{T}_{c,R}\rightarrow{}_{R}\mathpzc{Mod}$ the forgetful functor.

\begin{prop}[Corollary 1.2.8 \cite{dcfapp}]
Let $\mathcal{I}$ be a filtered category and $X:\mathcal{I}^{op}\rightarrow\mathcal{T}_{c,R}$ a functor. Then the following are equivalent.
\begin{enumerate}
\item
The map
$$\lim_{\mathcal{I}}X_{i}\rightarrow\mathbb{R}\lim_{\mathcal{I}}X_{i}$$
is an equivalence.
\item
$X$ satisfies condition SC and the map
$$\lim_{\mathcal{I}}|X_{i}|_{t}\rightarrow\mathbb{R}\lim_{\mathcal{I}}|X_{i}|_{t}$$
is an equivalence.
\end{enumerate}
\end{prop}

As a useful consequence we get the following.

\begin{cor}\label{cor:acycliclimtop}
Let $\mathcal{I}$ be a filtered category and $X:\mathcal{I}^{op}\rightarrow\mathcal{T}_{c,R}$ a functor.  If $|X|_{t}$ is weakly flasque then 
$$\lim_{\mathcal{I}}X_{i}\rightarrow\mathbb{R}\lim_{\mathcal{I}}X_{i}$$
is an equivalence.
\end{cor}

\begin{proof}
Such a system clearly satisfies condition SC. Moreover $|X|_{t}$ is a weakly flasque system in ${}_{R}\mathpzc{Mod}$. Thus $\lim_{\mathcal{I}}|X_{i}|_{t}\rightarrow\mathbb{R}\lim_{\mathcal{I}}|X_{i}|_{t}$ is an equivalence.
\end{proof}

\begin{cor}\label{cor:tcadepi}
The category $\mathcal{T}_{c,k}$ is ML.
\end{cor}

Thus we can derive the projective limit functor at the level of unbounded complexes. Moreover the functor $R_{PsN}^{T_{c}}:Pro(\sNorm_{R})\rightarrow\mathcal{T}_{c,R}$ is deformable at the level of unbounded complexes. This functor has a left adjoint $L^{PsN}_{T_{c}}$ which sends a locally convex module $X$ to $``\lim_{P_{X}}"X_{\rho}$, where $P_{X}$ is the family of continuous semi-norms on $X$, and $X_{\rho}$ is the semi-normed $R$-module $|X|_{t}$ equipped with the semi-norm $\rho$. By the characterisation of continuous maps and admissible epimorphisms of locally convex $R$-modules in terms of semi-norms in e.g. \cite{dcfapp} Corollary 1.1.8, it is straightforward to see that $L^{PsN}_{T_{c}}$ is exact.

%
%

\begin{lem}
The adjunction
$$\adj{L_{T_{c}}^{PsN}}{Ch(\mathcal{T}_{c,R})}{Ch(Pro(\sNorm_{k}))}{R_{PsN}^{T_{c}}}$$
is deformable, and for $X\in Ch(\mathcal{T}_{c,R})$, the map $\mathbb{R}R_{PsN}^{T_{c}}L_{T_{c}}^{PsN}(X)\rightarrow X$ is an equivalence.
\end{lem}

\begin{proof}
The functor $L_{T_{c}}^{PsN}$ is exact, so it is trivially deformable, and we explained above why $R_{PsN}^{T_{c}}$ is deformable.

Let $X\in Ch(\mathcal{T}_{c,k})$. First assume that $X$ is concentrated in degree $0$. Then $L_{T_{c}}^{PsN}(X)$ is $R_{PsN}^{T_{c}}$-acyclic by Corollary \ref{cor:acycliclimtop} and the dual of Proposition \ref{prop:indcolimiderived}. Indeed $\mathbb{R}R_{PsN}^{T_{c}}L_{T_{c}}^{PsN}(X)\cong\mathbb{R}\lim_{P_{X}}X_{\rho}$. But the system $\rho\mapsto |X_{\rho}|_{t}$ is clearly weakly flasque, since all the maps are the identity. Thus we have $\mathbb{R}R_{PsN}^{T_{c}}L_{T_{c}}^{PsN}(X)\cong R_{PsN}^{T_{c}}L_{T_{c}}^{PsN}(X)\cong X$. It follows that for any bounded above complex $X$, $\mathbb{R}R_{PsN}^{T_{c}}L_{T_{c}}^{PsN}(X)\rightarrow X$ is an equivalence. Now let $X$ be an arbitrary complex. The functor $L_{T_{c}}^{PsN}$ commutes with cokernels. Thus by Remark \ref{rem:weakerres}, $L_{T_{c}}^{PsN}(X)\cong\lim_{\omega^{op}}L_{T_{c}}^{PsN}(\tau^{R}_{\le n}X)$ is $R_{PsN}^{T_{c}}$-acyclic. Since $\mathcal{T}_{c,k}$ is ML it is in particular weakly $(\omega,\textbf{AdEpi})$-coelementary. Thus we get equivalences. 
$$X\cong\lim_{\omega^{op}}\tau^{R}_{\le n}X\cong R_{PsN}^{T_{c}}\lim_{\omega^{op}}L_{T_{c}}^{PsN}(\tau^{R}_{\le n}X)\cong \mathbb{R}R_{PsN}^{T_{c}}\lim_{\omega^{op}}L_{T_{c}}^{PsN}(\tau^{R}_{\le n}X)\cong \mathbb{R}R_{PsN}^{T_{c}}L_{T_{c}}^{PsN}(X)$$
\end{proof}

The situation for filtered colimits is less clear, but if $k$ is a Banach field then we have the following (which is proven for $k=\mathbb{C}$ in Proposition 1.1.11 \cite{dcfapp}).

\begin{prop}\label{prop:tcadmon}
The categories $\sNorm_{k}$ and $\mathcal{T}_{c,k}$ both have enough injectives. In particular $\mathcal{T}_{c,k}$ is weakly $(\omega,\textbf{AdMon})$-elementary. 
\end{prop}

\begin{proof}
By Lemma A.42 in \cite{koren} the category $\Ban_{k}$ has enough injectives. Now if $M\in\sNorm_{k}$, then $M\cong Sep(M)\oplus Zcl(M)$, where $Zcl(M)$ is the closure of $0$ in $M$ and is injective. Let $I$ be an injective Banach space and $Cpl(M)\rightarrow I$ an admissible monomorphism. $I$ remains injective in $\sNorm_{k}$, and the composite map $Sep(M)\rightarrow Cpl(M)\rightarrow I$ is an admissible monomorphism. Thus $M\rightarrow I\oplus Zcl(M)$ is an admissible monomorphism.  Let $X$ be a locally convex space, and let $P_{X}$ be the set of continuous semi-norms on $X$. There is an admissible monomorphism $X\rightarrow\prod_{p\in P_{X}}X_{p}$. For each semi-normed space $X_{p}$, pick an injective $I_{p}$ in $\sNorm_{k}$ and an admissible monomorphism $X_{p}\rightarrow I_{p}$. Then $X\rightarrow\prod_{p\in P_{X}}I_{p}$ is an admissible monomorphism, and $\prod_{p\in P_{X}}I_{p}$ is injective. 
\end{proof}

So for $k$ a Banach field we also have an unbounded derived colimit functor. 

\begin{cor}
The categories $\mathcal{T}_{c,k}$ and $Pro(\sNorm_{k})$ are both equipped with the injective model structure. Moreover the adjunction
$$\adj{L_{T_{c}}^{PsN}}{Ch(\mathcal{T}_{c,k})}{Ch(Pro(\sNorm_{k}))}{R_{PsN}^{T_{c}}}$$
is a Quillen adjunction.
\end{cor}

\begin{proof}
The claim regarding the injective model structures is an immediate consequence of \cite{kelly2016homotopy} Theorem 4.59 and Corollary \ref{cor:tcadepi}, and Proposition \ref{prop:tcadmon}. The fact that the adjunction is Quillen is immediate from the fact that $L_{T_{c}}^{PsN}$ is exact.
\end{proof}

\begin{prop}
 Let $\mathpzc{D}$ be an exact category, and let $F:\sNorm_{k}\rightarrow\mathpzc{D}$ be an exact functor. Suppose that $\mathpzc{D}$ is ML. Then the induced functor
 $$F^{t}:Ch(\mathcal{T}_{c,k})\rightarrow Ch(\mathpzc{D})$$
 is right-deformable. Moreover if $\mathpzc{D}$ is strongly ML then for $X\in\mathcal{T}_{c,k}$, $\mathbb{R}F^{t}(X)$ is concentrated in degree $0,-1$.
\end{prop}

\begin{proof}
The functors $F^{t}$ and $F^{t}\circ R_{PsN}^{T_{c}}$ are derivable by $K$-injective resolutions. Let $L^{PsN}_{T_{c}}(X)\rightarrow I_{\bullet}$ be an injective resolution in $Ch_{-}(Pro(\sNorm_{k}))$. Since $\mathbb{R}R_{PsN}^{T_{c}}L_{T_{c}}^{PsN}(X)\rightarrow X$ is an equivalence, $X\rightarrow R_{PsN}^{T_{c}}(I_{\bullet})$ is an injective resolution in $Ch_{-}(\mathcal{T}_{c,k})$. Thus $\mathbb{R}F^{t}(X)\cong\mathbb{R}(F^{t}\circ R_{PsN}^{T_{c}})(X)$. The claims now follow by Proposition \ref{prop:probounded}.
\end{proof}

\begin{rem}
The functor $(-)^{b}:\mathcal{T}_{c,k}\rightarrow\mathpzc{Born}_{k}$ has a left adjoint $(-)^{t}:\mathpzc{Born}_{k}\rightarrow\mathcal{T}_{c.k}$. This functor is also exact when restricted to the category of semi-normed spaces, and thus the functor
$$(-)^{t};Ch(\mathpzc{Born}_{k})\rightarrow Ch(\mathcal{T}_{c,k})$$
is left deformable.  By general nonsense, (\cite{dwyer2005homotopy} 44.2), there is an adjunction
$$\adj{\mathbb{L}(-)^{t}}{\Ho(Ch(\mathpzc{Born}_{k}))}{\Ho(Ch(\mathcal{T}_{c,k}))}{\mathbb{R}(-)^{b}}$$
In forthcoming work we will study this adjunction in greater detail.
\end{rem}

\bibliographystyle{amsalpha}
\bibliography{unbvarxiv2.bib}

\end{document}